\documentclass[reqno,10pt]{amsart}
\usepackage{amsmath,amsthm,amsfonts,amssymb}

\date{}
\newtheorem{theorem}{Theorem}[section]
\newtheorem{lemma}[theorem]{Lemma}
\newtheorem{corollary}[theorem]{Corollary}
\newtheorem{remark}[theorem]{Remark}

\newtheorem{example}[theorem]{Example}
\numberwithin{equation}{section}

\begin{document}

\vspace*{0.2cm}

\centerline{{\Large\bf An analytic approach to  infinite-dimensional}}

\vskip .1in

\centerline{{\Large\bf continuity and Fokker--Planck--Kolmogorov equations}}

\vspace*{0.1in}

\centerline{
 {\sc Vladimir I.~Bogachev, Giuseppe Da Prato,
Michael R\"ockner,}}

\centerline{{\sc and Stanislav V.~Shaposhnikov}}

\vskip 0.3in

{\small
\noindent
{\bf Abstract.}
 We prove a new uniqueness result for solutions to
Fokker--Planck--Kolmogorov  {\rm(}FPK{\rm)} equations for probability measures on
infinite-dimen\-si\-o\-nal spaces.
We consider infinite-dimensional drifts that admit certain
finite-dimensional approximations.
In contrast to most of the previous work on FPK-equations in infinite dimensions,
we include cases with non-constant coefficients in the second order
part and also include degenerate cases where these coefficients can even be
zero.
Also a new existence result is proved.
Some applications  to
Fokker--Planck--Kolmogorov  equations associated with SPDEs
are presented.
}

\vskip 0.2in



\noindent
{\bf Mathematics Subject Classification} \ 60H15, 60J35, 69J60, 47D07

\noindent
Running head: An analytic approach

\section*{Introduction}

In this paper we study the Cauchy problem
for infinite-dimensional Fokker--Planck--Kolmogorov equations of the form
$\partial_t \mu=L^{*}\mu$ for bounded
Borel measures $\mu$ on the space $\mathbb{R}^\infty\times (0, T_0)$,
where $\mathbb{R}^\infty$ is the countable power of $\mathbb{R}$ with the product topology, and
second order operators
$$
L\varphi=\sum_{i,j} a^{ij}\partial_{x_i}\partial_{x_j}\varphi+
\sum_{i} B^i \partial_{x_i}\varphi
$$
 defined on smooth functions of finitely many variables.
Such equations arise in many applications and
have been intensively studied in the last decades.
In particular, they are satisfied by transition probabilities of infinite-dimensional diffusions,
which is an important motivation for this paper. The finite-dimensional case has been studied
in depth by many authors (see the recent surveys \cite{BKR09} and \cite{BRS11b}),
in particular,
there is an extensive literature on regularity and uniqueness of solutions to
Fokker--Planck--Kolmogorov equations for measures on finite-dimensional spaces,
see \cite{BDR08},
\cite{BDPRS07}, \cite{BKR}, \cite{BKR09}, \cite{BRS11b}, \cite{Fig},
\cite{LBL}, \cite{SHDAN11},  and
the references there.
The infinite-dimensional case is considerably less studied, although there is also
a vast literature devoted to this case (see, e.g.,
\cite{BDPR09}, \cite{BDPR10}, \cite{BDPR11}, \cite{DD}, \cite{GN}, \cite{RS},
 and the references there).

The organization of the paper is as follows.
In Section~1 we introduce a general class
of Fokker--Planck--Kolmogorov equations
in infinite dimensions and prove some preliminary
results. In Section~2 we prove uniqueness of
probability solutions for these equations
under a certain approximative condition (which is
a condition on all components of the drift term
in a certain uniform way), which considerably generalizes our previous
uniqueness results in \cite{BDPR10} and \cite{BDPR11}.
The main difference with the finite-dimensional case is that in the latter
the global integrability of the coefficients $a^{ij}$ and $b^i$ with respect to the solution
ensures its uniqueness, but there is no infinite-dimensional analog of this simple
sufficient condition. What we prove is only a partial analog (Example~\ref{ex2.1}(ii)
formally gives a full analog, but the condition on the norm of the whole
drift is very restrictive in infinite dimensions).
More precisely, we establish two uniqueness results: Theorem~\ref{th1}
(nondegenerate diffusion matrices) and Theorem~\ref{th2} that applies also to
degenerate equations, in particular to fully degenerate transport (or continuity)
equations including the continuity equation associated to $2d$-Navier--Stokes equation.

In Section~3 we address the question of existence
of solutions  to our general FPK-equations and
prove Theorem~\ref{t4.1} which implies existence
under quite broad assumptions,
in particular, for stochastic Navier--Stokes equations over domains
in $\mathbb{R}^d$ for all dimensions~$d$.
In Section~2 and Section~3 we also consider examples that include
two other types of SPDEs, namely, stochastic
reaction diffusion equations on a bounded domain in $\mathbb{R}^d$
(Example~\ref{ex2.9})
and Burgers equation (Example~\ref{ex2.10})
on the interval $(0,1)$; their mixture is considered in Example~\ref{ex2.11}.

The approach and assumptions in this work differ from those in our earlier paper \cite{BDPR09},
where probabilistic tools were employed. Here we develop a purely analytic approach
without stochastic analysis and (for the first time in infinite dimension) also
include the case of nonconstant diffusion matrices.
The techniques are also different from the ones
in \cite{BDPR09}, \cite{BDPR10}, and \cite{DFPR}, where measures on Hilbert spaces were considered,
but the essential difference is not the type of infinite-dimensional spaces,
 but rather the method of proof which could be called
{\it approximative Holmgren method}, the idea of which is to multiply
the original equation by a solution of a certain equation approximating
the adjoint equation (but not the exact adjoint equation as in Holmgren's
method) and obtain after integration certain estimates (which replace
exact equalities in the classical Holmgren method).

This work was supported by the RFBR projects 13-01-00332, 
11-01-90421-Ukr-f-a, 11-01-12104-ofi-m, NS-8265.2010.1, 12-01-33009, the Russian
President Grant MD-754.2013.1,
and by the DFG through the program SFB 701 at the University of Bielefeld.

\section{Framework and preliminaries}

Let us describe our framework.
Let $B=(B^i(x, t))$ be a sequence of Borel
functions on
$\mathbb{R}^{\infty}\times(0, T_0)$,
where $T_0>0$ is fixed, and let
$a^{ij}$ be Borel functions on $\mathbb{R}^{\infty}\times(0,T_0)$.
Let us consider the Cauchy problem
\begin{equation}\label{e1}
\left\{\begin{array}{l}
\partial_t\mu=L^{*}\mu, \\
\mu|_{t=0}=\nu,
\end{array}\right.
\end{equation}
where $L^{*}$ is the formal adjoint operator for a differential operator $L$ defined by
$$
L\varphi(x, t)=\sum_{i, j=1}^{\infty}a^{ij}(x, t)\partial_{e_i}\partial_{e_j}
\varphi(x, t)
+\sum_{i=1}^{\infty}B^i(x, t)\partial_{e_i}\varphi(x, t)
$$
for every smooth function $\varphi$
depending on finitely many coordinates of $x$,
$\partial_{e_i}\varphi$ denotes the partial derivative along the vector
$e_i=(0,\ldots,0,1,0,\ldots)$.
Equations of this form are usually called
Fokker--Planck--Kolmogorov equations.

Throughout this paper a measure means a bounded signed measure (not necessarily nonnegative,
although our principal results will be concerned with probability measures).
The total variation of a measure $\mu$ is denoted by $|\mu|$.
Let $J$ be an interval in $[0, +\infty)$.
We use the standard notation $C(\mathbb{R}^k\times J)$ and
$C^{2,1}(\mathbb{R}^k\times J)$ for the class of real  continuous
functions on $\mathbb{R}^k\times J$ and its subclass
consisting of all functions $f$ having  continuous partial derivatives
$\partial_t f$, $\partial_{x_i}f$ and $\partial^2_{x_i x_j}f$.
Let $C_b(\mathbb{R}^k\times J)$ and
$C_b^{2,1}(\mathbb{R}^k\times J)$ denote the subclasses in these classes
consisting of bounded functions and functions $f$ with bounded derivatives
$\partial_t f$,
$\partial_{x_i}f$ and $\partial^2_{x_i x_j}f$,
respectively,
and $C_0^{2,1}(\mathbb{R}^k\times J)$
is the subspace in $C_b^{2,1}(\mathbb{R}^k\times J)$
consisting of functions with compact support
in $\mathbb{R}^k\times J$.

The inner product in $\mathbb{R}^d$
will be denoted by $\langle\,\cdot\,,\,\cdot\,\rangle$;
in the case of $L^2$-spaces we write $\langle\,\cdot\,,\,\cdot\,\rangle_2$
for its inner product
and the corresponding norm is denoted by $\|\,\cdot\,\|_2$.
The $L^p$-norm will be denoted by $\|\,\cdot\,\|_p$.
The norm $\|\,\cdot\,\|_{p,k}$
in the Sobolev space $H^{p,k}(U)$ of all functions
on a domain $U$ belonging to $L^p(U)$ along with
their generalized partial derivatives up to order $k$ is defined as
the sum of the $L^p$-norms of all partial derivatives up to order~$k$
(including $k=0$).

Let $P_N\colon\, \mathbb{R}^\infty\to\mathbb{R}^N$,
$P_Nx=(x_1,\ldots,x_N)$.
Given a function $\varphi$ on $\mathbb{R}^k$ we denote by the same
symbol the function on $\mathbb{R}^\infty$ defined by
$\varphi (x):=\varphi(P_kx)$.

We shall say that a bounded Borel  measure $\mu=\mu_t(dx)\,dt$ on
$\mathbb{R}^{\infty}\times(0, T_0)$,
where $(\mu_t)_{0<t<T_0}$  is a family of bounded Borel measures on $\mathbb{R}^{\infty}$,
satisfies the equation
$$
\partial_t\mu=L^{*}\mu
$$
if the functions  $a^{ij}$, $B^{i}$ are integrable with respect
to the variation $|\mu|$ of $\mu$
and for every $k\ge 1$
and every function $\varphi\in C^{2,1}_0(\mathbb{R}^k\times(0, T_0))$ we have
$$
\int_0^{T_0}\int_{\mathbb{R}^{\infty}}
\Bigl[\partial_t \varphi +
\sum_{i, j=1}^{\infty}a^{ij}\partial_{x_i}\partial_{x_j}\varphi
+\sum_{i=1}^{\infty}B^i\partial_{x_i}\varphi \Bigr] \,d\mu_t\,dt=0.
$$
It is obvious that it is enough to have this identity for all
$\varphi\in C^{\infty}_0(\mathbb{R}^k\times(0, T_0))$.

Let $\nu$ be a bounded Borel measure on $\mathbb{R}^{\infty}$. We say that the measure $\mu$ satisfies
the initial condition $\mu|_{t=0}=\nu$
if for every $k\ge 1$ and $\zeta\in C^{2}_0(\mathbb{R}^k)$ we have
$$
\lim_{t\to 0}\int_{\mathbb{R}^{\infty}}\zeta(x)\,\mu_t(dx)
=\int_{\mathbb{R}^d}\zeta(x)\,\nu(dx).
$$
Clearly, if $\sup_t \|\mu_t\|<\infty$, it suffices to have this equality
for all $\zeta\in C^{\infty}_0(\mathbb{R}^k)$.

We need the following auxiliary lemma.

\begin{lemma}\label{lem1.1}
Let $\mu=\mu_t(dx)\,dt$ be a solution to {\rm (\ref{e1})} such that $\sup_{t\in(0, 1)}\|\mu_t\|<\infty$.
Assume that $B^k\in L^1(|\mu|)$ for every $k\in\mathbb{N}$ and let
$0<T<T_0$. Then for every number $k\ge 1$
and every function $\varphi\in C_b(\mathbb{R}^k\times[0, T])
\bigcap C^{2,1}_b(\mathbb{R}^k\times (0, T))$
the  equality
\begin{equation}\label{e2}
\int_{\mathbb{R}^{\infty}}\varphi(x, t)\,\mu_t(dx)
=\int_{\mathbb{R}^{\infty}}\varphi(x, 0)\,\nu(dx)+
\int_0^t\int_{\mathbb{R}^{\infty}}
 [\partial_s\varphi+L\varphi]\,d\mu_s\,ds
\end{equation}
holds for almost every $t\in[0, T]$. Conversely, {\rm(\ref{e2})}
 implies {\rm(\ref{e1})}.
\end{lemma}
\begin{proof}
It is enough to prove this equality in the case where $\varphi(z, t)=0$
if $|z|>R>0$ for almost every $t\in[0, T]$.
Let $\eta\in C^{\infty}_0((0, T))$. According to our definition we have
$$
\int_0^{T}\int_{\mathbb{R}^{\infty}}
[\partial_t(\varphi\eta)+L(\varphi\eta)]\,d\mu_t\,dt=0.
$$
Thus, we obtain
$$
-\int_{0}^T
\eta'(t)\int_{\mathbb{R}^{\infty}}\varphi(x, t)\, \mu_t(dx)\,dt=
\int_{0}^T\eta(t)\int_{\mathbb{R}^{\infty}}
[\partial_t\varphi+L\varphi]\,d\mu_t\,dt.
$$
Hence the function
$$
t\mapsto\int_{\mathbb{R}^{\infty}}\varphi(x, t)\, \mu_t(dx)
$$
on $(0,T)$ has an absolutely continuous version for which
$$
\frac{d}{dt}\int_{\mathbb{R}^{\infty}}\varphi(x, t)\, \mu_t(dx)
=\int_{\mathbb{R}^{\infty}}
[\partial_t\varphi+L\varphi] \,d\mu_t.
$$
Therefore, for some constant $C\in\mathbb{R}$ the equality
$$
\int_{\mathbb{R}^{\infty}}\varphi(x, t)\, \mu_t(dx)
=C+\int_0^t\int_{\mathbb{R}^{\infty}}
[\partial_s\varphi+L\varphi]\,\mu_s\,ds
$$
holds for almost every $t\in[0, T]$.
Note that $\varphi(x, t)$  converges uniformly to $\varphi(x, 0)$ as $t\to 0$. Moreover, we have
$$
\lim_{t\to 0}\int_{\mathbb{R}^{\infty}}\varphi(x, 0)\,\mu_t(\,dx)=
\int_{\mathbb{R}^{\infty}}\varphi(x, 0)\,\nu(dx).
$$
It follows that
$$
C=\int_{\mathbb{R}^{\infty}}\varphi(x, 0)\, \nu(dx),
$$
which completes the proof of one implication. The converse is, however, obvious.
\end{proof}

\begin{remark}
\rm
Let $k\in\mathbb{N}$.
If $\varphi(\, \cdot \,, t)=\psi\in C^2_b(\mathbb{R}^k)$ for every $t\in[0, T]$, $T<T_0$,
then by (\ref{e2}) we have
\begin{equation}\label{e3}
\int_{\mathbb{R}^{\infty}}\psi(x)\, \mu_t(dx)
=\int_{\mathbb{R}^{\infty}}\psi(x)\, \nu(dx)+
\int_0^t\int_{\mathbb{R}^{\infty}}L\psi(x, s)\, \mu_s(dx)\,ds
\end{equation}
for almost all $t\in [0, T]$.
Moreover, if $J^{\mu}_{\psi}$ denotes the set of all $t\in[0, T]$ such that equality (\ref{e3}) holds,
then the closure of $J^{\mu}_{\psi}$ coincides with $[0, T]$ and the restriction
of the mapping
$$
t\mapsto\int_{\mathbb{R}^{\infty}}\psi(x)\,\mu_t(dx)
$$
 to $J^{\mu}_{\psi}$ is continuous, since the right-hand side of (\ref{e3}) is continuous in~$t$.
\end{remark}

\begin{remark}\label{rem1.3}
\rm
Let $\varphi$ be as in Lemma 1.1. and assume that
$T\in J^{\mu}_{\varphi(\, \cdot \,, T)}$. Then equality (\ref{e2}) holds with $t=T$. Indeed,
$\varphi(x, t)$  converges uniformly to $\varphi(x, T)$ as $t\to T$.
Let $I$ be the set of all $t\in[0, T]$ such that equality (\ref{e2}) holds.
Let us take a sequence ${t_n}\in J^{\mu}_{\varphi(\,\cdot\,, T)}\bigcap I$ such that
$\lim\limits_{n\to\infty}t_n=T$. Then we have
$$
\lim_{n\to\infty}\int_{\mathbb{R}^{\infty}}
\varphi(x, t_n)\, \mu_{t_n}(dx)
=\int_{\mathbb{R}^{\infty}}\varphi(x, T)\, \mu_T(dx)
$$
and equality (\ref{e2}) holds for each $t_n$. Letting $n\to\infty$, we obtain equality (\ref{e2}) with $t=T$.
\end{remark}

\section{Uniqueness of probability solutions}

In this section two establish two different uniqueness results: first we consider
nondegenerate diffusion matrices and then turn to the general case that includes
fully degenerate equations. We start with stating our assumptions about $A$ and $b$.

\, (A) \,
$a^{ij}=a^{ji}$, each function $a^{ij}$
 depends only on the variables
$t, x_1, x_2, \ldots, x_{\max\{i, j\}}$ and is continuous and for every natural number $N$
the matrix $A_N=(a^{ij})_{1\le i, j\le N}$
satisfies the following condition:

there exist positive numbers $\gamma_N$, $\lambda_N$ and $\beta_N\in (0, 1]$ such that
for all $x, y\in\mathbb{R}^N$ and $t\in [0, T_0]$ one has
$$
\gamma_N|y|^2\le \langle A_N(x, t)y, y\rangle \le\gamma_N^{-1}|y|^2, \quad
\|A_N(x, t)-A_N(y, t)\|\le \lambda_N|x-y|^{\beta_N},
$$
where $\|\,\cdot\,\|$ is  the operator norm and $|\,\cdot\,|$ is the standard Euclidean norm.

Let $\nu$ be a Borel probability measure on $\mathbb{R}^{\infty}$ and
let $\mathcal{P}_{\nu}$ be some {\it convex} set of probability solutions $\mu=\mu_t(dx)\,dt$ to (\ref{e1}),
i.e., $\mu_t\ge 0$ and $\mu_t(\mathbb{R}^{\infty})=1$ for every
$t\in (0, T_0)$,
such that $|B^k|\in L^2(\mu)$ for each $k\in\mathbb{N}$ and the following condition holds:

\, (B) \, for every $\varepsilon>0$ and every natural number $d$ there exist a natural number $N\ge d$
and a $C^{2,1}_b$-mapping $(b^k)_{k=1}^N\colon\,
\mathbb{R}^N\times [0, T_0]\to \mathbb{R}^N$ such that
$$
\int_0^{T_0}\int_{\mathbb{R}^\infty}
|A_N(x,t)^{-1/2}(B_N(x, t)-b(x_1, \ldots, x_N, t))|^2\,
\mu_t(dx)\,dt<\varepsilon,
$$
where $B_N=(B^1, \ldots, B^N)$.

Let us illustrate condition (B) by several examples.

\begin{example}\label{ex2.1}
{\rm
\, (i) \, Let $B^k$ depend only on the variables $t, x_1, x_2, \ldots, x_k$.
Then in order to ensure our condition (B) we need
only the inclusion $|B^k|\in L^2(\mu)$
for all $k\ge 1$.
Indeed, we set $N=d$ and approximate each function $B^k$ separately.

\, (ii) \,
Let $\alpha_k$ be a positive number for each $k\in\mathbb{N}$ and
$$
l_{1/\alpha}^2=\Bigl\{(z_k)\colon \, \sum_{k=1}^{\infty}\alpha_k^{-1}z_k^2<\infty\Bigr\},
\quad
\|x\|_{1/\alpha}=\Bigl(\sum_{k=1}^{\infty}\alpha_k^{-1}z_k^2\Bigr)^{1/2}.
$$
Suppose that $a^{ij}$ satisfy condition (A) and there exists a positive number $C$ independent of $N$
such that
$$
|A_N(x, t)^{-1/2}y|\le C\|y\|_{l_{1/\alpha}}
$$
for all $x$, $t$ and $y=(y_1, y_2, \ldots, y_N, 0, 0, \ldots)$.
For example, this is true if $a^{ij}=0$ for $i\neq j$ and $a^{ii}=\alpha_i$.

Let $(B^k(x, t))\in l_{1/\alpha}^2$ for $\mu$-almost every $(x, t)$ and let $\|B\|_{1/\alpha}\in L^2(\mu)$.
For every $\varepsilon>0$ and every natural number $d$ we pick  a number $M>d$ such that
$$
\sum_{k=M+1}^{\infty}\int_0^{T_0}
\int_{\mathbb{R}^{\infty}}\alpha_k^{-1}|B^k|^2\,d\mu_t\,dt<\varepsilon/2.
$$
Then for every $B^k$ we find a  smooth function $b^k$ depending on the first $n_k$ variables such that
$$
\int_0^{T_0}\int_{\mathbb{R}^{\infty}}
\alpha_k^{-1}|B^k-b^k|^2\,d\mu_t\,dt< \varepsilon(2M)^{-1},
\quad k=1,\ldots,M.
$$
Set $N=\max\{M, n_1, n_2, \ldots, n_{M}\}$ and $b^k\equiv 0$ for $k>M$. Then
\begin{multline*}
\sum_{k=1}^{N}\int_0^{T_0}\int_{\mathbb{R}^{\infty}}\alpha_k^{-1}|B^k-b^k|^2\,d\mu_t\,dt
\\
=\sum_{k=1}^{M}\int_0^{T_0}\int_{\mathbb{R}^{\infty}}\alpha_k^{-1}|B^k-b^k|^2\,d\mu_t\,dt+
\sum_{k=M+1}^{N}\int_0^{T_0}\int_{\mathbb{R}^{\infty}}\alpha_k^{-1}|B^k|^2\,d\mu_t\,dt<\varepsilon.
\end{multline*}

\, (iii) \, Finally, for $a^{ij}$ as in (ii),
we can combine both examples. Let $B=G+F$,
where $G^k, F^k\in L^2(\mu)$,
$G^k(x,t)=G^k(x_1, x_2, \ldots, x_k, t)$,
$F(x, t)\in l^2_{1/\alpha}$
and $\|F\|_{1/\alpha}\in L^2(\mu)$.
Obviously, for given $B^k$ of this type the set of all
probability solutions $\mu=\mu_t(dx)dt$
to (\ref{e1}) satisfying the previous
integrability conditions is convex.
}\end{example}

\begin{remark}\rm
(i) Obviously, condition (B) is equivalent to the following: there exist an increasing
sequence $N_l\to +\infty$ and $C_b^{2,1}$-mappings $b^l=(b^{l,k})_{k=1}^{N_l}$
on $\mathbb{R}^{N_l}\times [0,T_0]$ such that
$$
\lim\limits_{l\to\infty}
\int_0^{T_0}\int_{\mathbb{R}^\infty}
|A_{N_l}(x,t)^{-1/2}(B_{N_l}(x, t)-b^l(x_1, \ldots, x_{N_l}, t))|^2\,
\mu_t(dx)\,dt=0.
$$

(ii)
Assume that $a^{ij}=\delta^{ij}$.
Let $\widetilde{P}_N(x,t)=(P_Nx,t)$
and let $\mathbb{E}_{\mu}[\, \cdot \,|\widetilde{P}_N=(x, t)]$
be the corresponding conditional expectation. Then condition (B) is equivalent to the
following:
for every $\varepsilon>0$ and every natural number $d$ there exists a natural number $N\ge d$
such that
$$
\int_0^{T_0}\int_{\mathbb{R}^\infty}
\sum_{k=1}^{N}\Bigl|B^k(x, t)-\mathbb{E}_{\mu}
[B^k |\widetilde{P}_N=(x, t)]\Bigr|^2\,\mu_t(dx)\,dt<\varepsilon.
$$
This condition is known in Euclidian quantum field theory as the H{\o}egh-Krohn condition
 (see~\cite{AH-K77})
and has been used, e.g., to prove Markov uniqueness for semigroups (see \cite{RZ92}).
\end{remark}

\begin{theorem}\label{th1}
Assume that condition {\rm (A)} holds. Then the set $\mathcal{P}_{\nu}$ contains at most one element.
\end{theorem}
\begin{proof}
Assume that two measures $\sigma^1=\sigma^1_tdt$ and $\sigma^2=\sigma^2_tdt$ belong to $\mathcal{P}_{\nu}$.
By our assumption about $\mathcal{P}_\nu$,
$\sigma=(\sigma^1+\sigma^2)/2\in\mathcal{P}_{\nu}$.
Let $d\in\mathbb{N}$, $\psi\in C^{\infty}_0(\mathbb{R}^d)$ and $|\psi(x)|\le 1$ for all $x\in\mathbb{R}^d$.
By condition~(B) for every $\varepsilon>0$
there exist a natural number $N\ge d$ and
a $C^{2, 1}_b$-mapping $(b^k)_{k=1}^N$ on
$\mathbb{R}^N\times [0, T_0]$ such that
$$
\int_0^{T_0}\int_{\mathbb{R}^\infty}|A_N^{-1/2}(x,s)(B_N(x, s)-b(x_1, \ldots, x_N, s))|^2\,
\sigma_s(dx)\,ds<\varepsilon.
$$
Fix
$t\in J^{\sigma^1}_{\psi}\bigcap J^{\sigma^2}_{\psi}\bigcap J^{\sigma^1}_{\psi^2}\bigcap J^{\sigma^2}_{\psi^2}$.
Let $f$ be a solution to the finite-dimensional Cauchy problem
\begin{equation}\label{e2.1}
\left\{\begin{array}{l}
\partial_tf+\sum_{i, j=1}^{N}a^{ij}\partial_{x_i}\partial_{x_j}f
+\sum_{i=1}^Nb^i\partial_{x_i}f=0 \quad \hbox{on $\mathbb{R}^N\times (0,t)$}, \\
f(t, x)=\psi(x).
\end{array}\right.
\end{equation}
It is  known (see, e.g., \cite[Theorem~1.3]{ED1} and also \cite{ED2}, \cite{Fried}, and \cite{Str})
that a solution exists and belongs to
the class $C_b(\mathbb{R}^N \times [0, t])
\bigcap C^{2,1}_b(\mathbb{R}^N \times (0, t))$.
Moreover, according to the maximum principle $|f(x, s)|\le 1$
for all $(x,s)\in \mathbb{R}^N\times [0,t]$.
Set $\mu=\sigma^1-\sigma^2$.
The measure $\mu$ solves the Cauchy problem (\ref{e1})
with zero initial condition.
Applying Lemma \ref{lem1.1} and Remark \ref{rem1.3} with $\varphi=f$, we obtain
$$
\int_{\mathbb{R}^{\infty}}f(x, t)\, \mu_t(dx)=
\int_0^t\int_{\mathbb{R}^{\infty}}
\Bigl[ \partial_s f
+\sum_{i, j=1}^{N}a^{ij}\partial_{x_j}\partial_{x_i}f
+\sum_{i=1}^NB^i\partial_{x_i}f\Bigr]
\, d\mu_s\,ds.
$$
Therefore,
\begin{equation}\label{e2.2}
\int_{\mathbb{R}^{\infty}}\psi
\,d\mu_t=\int_0^t\int_{\mathbb{R}^{\infty}}
\langle B-b, \nabla f\rangle\,d\mu_s\,ds.
\end{equation}
Let us estimate the following expression:
$$
\int_0^t\int_{\mathbb{R}^{\infty}}|\sqrt{A_N}\nabla f|^2\, d\sigma_s\,ds.
$$
Using (\ref{e2}) for $\sigma$ and $\varphi=f^2$, taking into account that
$(\partial_s+L)(f^2)=2f(\partial_s+L)f+2|\sqrt{A_N}\nabla f|^2$,
 and recalling that
$t\in J^{\sigma^1}_{\psi^2}\bigcap J^{\sigma^2}_{\psi^2}$, we obtain from (\ref{e2.1})
(again by Remark~\ref{rem1.3}) that
$$
\int_{\mathbb{R}^{\infty}}\psi^2\,d\sigma_t
-\int_{\mathbb{R}^{\infty}} f^2(x, 0)\, \nu(dx)=
2\int_0^t\int_{\mathbb{R}^{\infty}}
\Bigl[
|\sqrt{A_N}\nabla f|^2+f\sum_{i=1}^{N}(B^i-b^i)\partial_{x_i}f\Bigr]
\,d\sigma_s\,ds.
$$
Therefore,
$$
\int_0^t\int_{\mathbb{R}^{\infty}}|\sqrt{A_N}\nabla f|^2\,d\sigma_s\,ds\le
2+\int_0^{T_0}
\int_{\mathbb{R}^\infty}|A_N^{-1/2}(x,s)(B_N(x, s)-b(x_1, \ldots, x_N, s))|^2\,
\sigma_s(dx)\,ds.
$$
Thus we obtain the estimate
\begin{equation}\label{e2.3}
\int_0^t\int_{\mathbb{R}^{\infty}}|\sqrt{A_N}\nabla f|^2\,d\sigma_s\,ds\le 2+\varepsilon.
\end{equation}
Applying (\ref{e2.2}) and (\ref{e2.3}) and the fact that $|\mu|\le\sigma^1+\sigma^2=2\sigma$ we have
$$
\int_{\mathbb{R}^{\infty}}\psi\,d\mu_t\le 2\sqrt{\varepsilon(2+\varepsilon)}.
$$
Since $\varepsilon>0$ was arbitrary, we obtain
$$
\int_{\mathbb{R}^{\infty}}\psi\,d\mu_t\le 0.
$$
Replacing $\psi$ with $-\psi$ we arrive at the equality
$$\int_{\mathbb{R}^{\infty}}\psi\,d\mu_t=0.
$$
Therefore,
$$
\int_{\mathbb{R}^{\infty}}
\psi\,d\sigma^1_t=\int_{\mathbb{R}^{\infty}}\psi\,d\sigma^2_t
$$
for every
$t\in J^{\sigma^1}_{\psi}\bigcap J^{\sigma^2}_{\psi}\bigcap J^{\sigma^1}_{\psi^2}\bigcap J^{\sigma^2}_{\psi^2}$,
hence for almost every $t\in[0, T_0]$. Thus, $\sigma^1=\sigma^2$.
\end{proof}

We now consider a typical example to which the previous theorem applies,
namely, the
Fokker--Planck--Kolmogorov equations associated with
stochastic partial differential equations of  reaction diffusion type
on a domain $D\subset \mathbb{R}^d$, i.e.,
 $$
du(t)=\sigma(u(t),t)dW(t)+B(u(t),t)dt, \ t\in [0,T_0],
$$
where $\sigma\sigma^{*}=A$ and $u(t)\in L^2(D)$. Furthermore, $W(t)$, $t\ge 0$,
is a cylindrical Wiener process in $L^2(D)$ on a stochastic basis
$(\Omega, \mathcal{F}, (\mathcal{F}_t), {\rm P})$ and $u(0)$ has $\nu$ as given law.
Below we denote by $u$ generic elements of functional spaces such as $L^2(D)$
 which we embed into $\mathbb{R}^\infty$
 (e.g., by using a suitable orthonormal basis)
 to be able to apply our framework above.

\begin{example}\label{ex2.4}
{\rm
(``Reaction diffusion equations in dimension $d$ with infinite trace'')
Suppose that $D\subset\mathbb{R}^d$ is an open bounded set and $\{e_k\}$
is an eigenbasis of  the Laplacian
on $L^2(D)$ with zero boundary condition, i.e., $\Delta e_k=-\lambda_k^2e_k$.
Let $f\colon\, D\times\mathbb{R}\times [0,T_0]\to \mathbb{R}$ be a Borel function.
Set $B(u,t)(z)=\Delta u(z)+f(z,u(z),t)$, $z\in D$, i.e.,
$$
B^i(u,t)=-\lambda_i^2u_i+\langle f(\,\cdot\,,u(\cdot),t), e_i\rangle_{2},
\ u\in L^2(D), \ u_i=\langle u,e_i\rangle_{2}.
$$
Assume that the coefficients $a^{ij}$ satisfy (A) with $\gamma_N=\gamma>0$ independent on $N$.
For instance, the last assumption is true if
$a^{ij}=\langle Se_i, e_j\rangle_{2}$
for some invertible symmetric positive operator $S$ on $L^2(D)$.

Assume also that there exist a Borel function $C\ge 0$
and a number $m\ge 1$ such that
$$
|f(z, u,t)|\le C(t)+C(t)|u|^m.
$$
Set
$$
L\varphi=\sum_{i, j=1}^{\infty}a^{i j}\partial_{e_i}\partial_{e_j}\varphi
+\sum_{i=1}^{\infty}B^i\partial_{e_i}\varphi.
$$
Then there is at most one probability solution $\mu=\mu_t(du)\,dt$,
i.e., $\mu_t\ge 0$ and $\mu_t(\mathbb{R}^{\infty})=1$ for every
$t\in (0, T_0)$, to the Cauchy problem (\ref{e1}) such that
$$
\int_0^{T_0} C(t)^2\int_{L^2(D)}\|u\|_{2m}^{2m}\,\mu_t(du)\,dt<\infty.
$$
}\end{example}
\begin{proof}
The mapping $u\mapsto (u_i)$ defines an embedding $L^2(D)\to\mathbb{R}^\infty$.
Extending $B^i$ and $a^{ij}$ to all of $\mathbb{R}^\infty \times  [0, T_0]$
by zero we end up
in the framework described above.
Set $F^i(u,t)=\langle f(\,\cdot\,,u(\cdot),t), e_i\rangle_{2}$.
Note that
$$
\sum_{i=1}^{\infty}|F^i(u,t)|^2=
\|f(\,\cdot\, ,u(\cdot),t)\|_{L^2}^2\le C(t)^2+C(t)^2\|u\|_{2m}^{2m}.
$$
Thus we have $B^i=A^i+F^i$, where $A^i(u)=-\lambda_i^2 u_i$
and $\|F\|_{l^2}\in L^2(\mu)$,
and Example~\ref{ex2.1}(iii) applies with $\alpha_k=1$.
\end{proof}

Let now $d=1$, $D=(0, 1)$ and $\Delta=\frac{d^2}{dz^2}$.
We recall that according to  \cite{BDPR10} and \cite{BDPR11}
if $a^{ij}=\alpha \delta^{ij}$ with $\alpha>0$ and if
$$
f(z, u,t)=f_1(z, u,t)+f_2(z, u,t),
$$
where $(u,t)\mapsto f_i(z,u,t)$ are continuous for each~$z$
and for some nonnegative functions $c_1, c_3\in L^2((0, T_0))$, $c_2\in L^1((0, T_0))$
and all $t, z, u$ we have

\, (i) \, $|f_1(z, u,t)|\le c_1(t)(1+|u|^m)$,

\, (ii) \, $(f_1(z, u,t)-f_1(z, v,t))(u-v)\le c_2(t)|u-v|^2$,

\, (iii) \, $|f_2(z, u,t)|\le c_3(t)(1+|u|)$,

\noindent
then for every initial value $\nu$ with $\|u\|_{2m}^{2m}\in L^1(\nu)$
there exists a probability solution $\mu$ of the Cauchy problem (\ref{e1})
such that $(1+c_1(t)+c_3(t))^2(1+\|u\|_{2m}^{2m})\in L^1(\mu)$.
It follows from the previous example that such a solution is unique, which
improves the uniqueness result from \cite{BDPR10} and \cite{BDPR11}.

\vskip .1in

We now present another uniqueness condition that applies to degenerate (even zero)
diffusion matrices.
Let us list our new assumptions (${\rm A}'$) and (${\rm B}'$).

\, (${\rm A}'$) \, $A(x,t)=(a^{ij}(x,t))$, where
each function $a^{ij}$ is bounded and
depends only on the variables
$x_1, x_2, \ldots, x_{\max\{i, j\}},t$ and for every natural number $N$
the matrix $A_N$ is symmetric nonnegative and
the elements $\sigma^{ij}_N$ of the matrix $\sigma_N:=\sqrt{A_N}$
are in the class $C^{\infty}(\mathbb{R}^N\times[0, T_0])$.

Let $\nu$ be a Borel probability measure on $\mathbb{R}^{\infty}$ and
let $\mathbb{P}_{\nu}$ be some {\it convex} set of probability solutions $\mu=\mu_t(dx)\,dt$ of (\ref{e1}),
i.e., $\mu_t\ge 0$ and $\mu_t(\mathbb{R}^{\infty})=1$ for every
$t\in (0, T_0)$,
such that $|B^k|\in L^1(\mu)$ for each $k\in\mathbb{N}$ and the following condition holds:

\, (${\rm B}'$) \, for every $\varepsilon>0$ and every natural number $d$ there exist a natural number $N\ge d$,
a $C^{\infty}$-mapping $b=(b^k)_{k=1}^N\colon\,
\mathbb{R}^N\times [0, T_0]\to \mathbb{R}^N$, a function $\theta$ on $\mathbb{R}^N$,
a function $V\in C^2(\mathbb{R}^N)$ with $V\ge 1$,
and numbers $C_0\ge 0$ and $\delta>0$ such that

\, (i) \, $\sqrt{V(P_Nx)}$, $|B_N(x, t)-b(P_Nx, t)|\sqrt{V(P_Nx)}\in L^1(\mu)$
and
$$
\int_0^{T_0}\int_{\mathbb{R}^\infty}
|B_N(x, t)-b(P_Nx, t)|\sqrt{V(P_Nx)}e^{C_0(T_0-t)/2}\,
\mu_t(dx)\,dt<\varepsilon,
$$
where $B_N=(B^1, \ldots, B^N)$;

\, (ii) \, the matrix $\mathcal{B}=(\partial_{x_j}b^i)$ and the operator
$$
L_{a, b}\varphi(x, t)=\sum_{i,j\le N} a^{ij}(x,t)\partial_{x_i}\partial_{x_j}\varphi(x, t)+
\sum_{i\le N} b^i(x, t)\partial_{x_i}\varphi(x, t)
$$
satisfy the estimates
$$
\langle \mathcal{B}(x, t)h, h\rangle \le \theta(x)|h|^2 \ \ \forall\, h\in\mathbb{R}^N, \quad
L_{a, b}V(x, t)\le (C_0-\Lambda(x, t))V(x),
$$
where
$$
\Lambda(x, t):=
4\sum_{i, j, k\le N}\Bigl|\partial_{x_k}\sigma_N^{ij}(x, t)\Bigr|^2
+2\theta(x)+\delta(1+|x|^2)^{-1}|b(x, t)|^2)
$$
for every $(x, t)\in \mathbb{R}^N\times[0, T_0]$.

In the notation for $N,b,\theta,V,C_0,\delta$ we omit indication of the fact that they depend
on $\varepsilon$ and $d$. Recall also that $|\,\cdot\,|$ is the standard Euclidean norm.

\begin{theorem}\label{th2}
If {\rm$({\rm A}')$} holds, then the set $\mathbb{P}_{\nu}$ contains at most one element.
\end{theorem}

\begin{remark}\rm
\, (i) \, If $A=(a^{ij})$ is a constant matrix and $|b(x, t)|\le C_1(N)+C_1(N)|x|$,
then condition~(${\rm B}'$)(ii) can be replaced by
$$
L_{a, b}V(x, t)\le (C_0-2\theta(x))V(x)\quad \forall\, (x, t)\in \mathbb{R}^N\times[0, T_0].
$$

\, (ii) \,
If $A=(a^{ij})$ is a constant matrix and $|b(x, t)|\le C_1(N)+C_1(N)|x|^2$,
then condition~(${\rm B}'$)(ii) can be replaced by
$$
L_{a, b}V(x, t)\le (C_0-2\theta(x)-\delta|x|^2)V(x)
$$
for every $(x, t)\in \mathbb{R}^N\times[0, T_0]$ and some $\delta>0$.

\, (iii) \, Let $a^{ij}=0$ if $i\ne j$ and $a^{ii}(x, t)=\alpha^i(x_1, x_2, \ldots, x_i, t)\ge 0$.
Suppose also that we have
 $|b(x, t)|\le C_1(N)+C_1(N)|x|^2$. Then condition~(${\rm B}'$)(ii) can be replaced by
$$
L_{a, b}V(x, t)\le (C_0-\Lambda(x, t))V(x), \quad
\Lambda(x, t):=4\sum_{i=1}^{N}\sum_{k\le i}\frac{\bigl|\partial_{x_k}\alpha^{i}(x, t)\bigr|^2}{\alpha^i(x, t)}+
2\theta(x)+\delta|x|^2.
$$

\, (iv)\, We note that (${\rm B}'$) is a substantial generalization of
a corresponding condition in \cite{RS}.
\end{remark}

Let us illustrate condition (${\rm B}'$).

For a sequence $\{\lambda_k^2\}_k$, we write
$$
\|x\|_{l^2_{\lambda}}^2=\sum_k\lambda_k^2x_k^2, \quad
(x, y)_{l^2_{\lambda}}=\sum_k\lambda_k^2x_ky_k.
$$

\begin{example}\label{ex2.7}
\rm
We assume here that $A=(a^{ij})_{i,j\ge1}$ is a {\it constant} matrix,
$A_N:=(a^{ij})_{i,j\le N}$ is symmetric nonnegative.

\, (i) \, Let $b^{k}(x, t)=-\lambda_k^2x_k+f^k(x, t)$, $x\in \mathbb{R}^N$.
Then the estimate $\langle \mathcal{B}h, h\rangle\le \theta(x)|h|^2$,
$x,h\in \mathbb{R}^N$,
follows from the estimate
$$
\langle\mathcal{F}(x, t)h, h\rangle\le \theta(x)|h|^2+\|h\|_{l^2_{\lambda}}^2,
\quad x,h\in \mathbb{R}^N,
$$
where $\mathcal{F}=(\partial_{x_j}f^i)_{i,j\le N}$.

\, (ii) \,
Set $V(x)=\exp\bigl(\kappa\sum_{k=1}^Nx_k^2\bigr)$, where $\kappa>0$.
Then the condition on $\theta$ required in (${\rm B}'$) is this:
for some numbers $C_0$ and $\delta>0$ (dependent on $\varepsilon$ and~$d$)
one has
\begin{equation}\label{gener}
\theta(x)\le C_0-\kappa\bigl({\rm tr} A_N
+2\kappa\langle A_N x, x\rangle +\langle b(x, t), x\rangle\bigr)-2^{-1}\delta(1+|x|^2)^{-1}|b(x, t)|^2,
\quad x\in\mathbb{R}^N.
\end{equation}
Let us consider a more specific case:
$b^{k}(x, t)=-\lambda_k^2x_k+f^k(x, t)$, $f(x,t)=(f^k(x,t))_{k=1}^N$,
$\langle f(x, t), x\rangle\le 0$ and $|f^k(x, t)|\le C_1+C_2|x|^2$, where $x\in\mathbb{R}^N$.
Assume that for some $\varepsilon_0>0$ and every $N\ge 1$ one has
$$
\varepsilon_0\bigl(\langle A_Nx, x\rangle +|x|^2\bigr)\le \|x\|_{l^2_{\lambda}}^2,
\quad x\in\mathbb{R}^N.
$$
Then condition (${\rm B}'$)(ii) can be rewritten in the following form:
$$
\langle\mathcal{F}(x, t)h, h\rangle\le \theta(x)|h|^2+\|h\|_{l^2_{\lambda}}^2,
\quad x,h\in \mathbb{R}^N,
$$
where $\mathcal{F}=(\partial_{x_j}f^i)_{i,j\le N}$,
and for every $x\in\mathbb{R}^N$
$$
\Theta(x)\le C_0-\kappa{\rm tr} A_N+2^{-1}\kappa(\varepsilon_0-\kappa)\|x\|_{l^2_{\lambda}}^2.
$$
Note that in this case we take $V(x)$ with $\kappa<\varepsilon_0/4$.

This assertion follows from (\ref{gener}) if we choose $\delta>0$ such that
$$
\delta(1+|x|^2)^{-1}|b(x, t)|^2\le \varepsilon_0\kappa|x|^2+1.
$$

\, (iii) \,
Let $V(x)=\exp\bigl(\kappa\|x\|^2_{l^2_{\lambda}})$.
Then the condition on $\theta$ required in (${\rm B}'$) is this:
for some constants $C_0$ and $\delta>0$ one has
\begin{equation}\label{gener1}
\theta(x)\le C_0-\kappa\Bigl(\sum_{i=1}^N a^{ii}\lambda^2_i
+2\kappa\sum_{i, j\le N} a^{ij}\lambda^2_i\lambda^2_jx_ix_j+\langle b(x, t), x\rangle_{l^2_{\lambda}}\Bigr)
-2^{-1}\delta(1+|x|^2)^{-1}|b(x, t)|^2,
\quad x\in \mathbb{R}^N.
\end{equation}
Let us consider a more specific case:
$b^{k}(x, t)=-\lambda_k^2x_k+f^k(x, t)$, $f(x,t)=(f^k(x,t))_{i=1}^N$,
$\langle f(x, t), x\rangle_{l^2_{\lambda}}\le 0$ and
$|f^k(x, t)|\le C_1+C_2|x|^2$, where $x\in\mathbb{R}^N$.
Assume that for some $\varepsilon_0>0$ and every $N\ge 1$ one has
$$
\varepsilon_0\sum_{i,j\le N}a^{ij}\lambda^2_i\lambda^2_jx_ix_j+\varepsilon_0|x|^2\le \sum_{i\le N}
\lambda_i^4x_i^2.
$$
Then condition (${\rm B}'$)(ii) can be rewritten in the following form:
$$
\langle\mathcal{F}(x, t)h, h\rangle\le \theta(x)|h|^2+\|h\|_{l^2_{\lambda}}^2,
\quad x,h\in \mathbb{R}^N,
$$
where $\mathcal{F}=(\partial_{x_j}f^i)_{i,j\le N}$,
and for every $x\in\mathbb{R}^N$
$$
\theta(x)\le C_0-\kappa\sum_{i=1}^N a^{ii}\lambda^2_i+2^{-1}\kappa(\varepsilon_0-\kappa)\sum_{i\le N}\lambda_i^4x_i^2.
$$
Note that in this case we take $V(x)$ with $\kappa<\varepsilon_0/4$.

This assertion follows (\ref{gener1})
if we choose $\delta>0$ such that
$$
\delta(1+|x|^2)^{-1}|b(x, t)|^2\le \varepsilon_0\kappa|x|^2+1.
$$
\end{example}

For the proof of Theorem~\ref{th2} we need the following lemma.

Let  $\eta\in C^{\infty}_0(\mathbb{R}^1)$ be such that
$\eta(x)=1$ if $|x|\le 1$ and $\eta(x)=0$ if $|x|>2$,
$0\le\eta\le 1$ and there exists a number $C>0$ such that $|\eta'(x)|^2\eta^{-1}(x)\le C$
for every $x$.

\begin{lemma}\label{lem2.8}
Assume that there exist
a function $\theta$ on $\mathbb{R}^N$, a function $V\in C^2(\mathbb{R}^N)$ with
$V\ge 1$, and numbers $C_0\ge 0$ and $\delta>0$
such that for all $(x, t)\in \mathbb{R}^N\times[0, T_0]$, $h\in\mathbb{R}^N$ one has
$$
\langle \mathcal{B}(x, t)h, h\rangle
\le \theta(x)|h|^2, \quad \mathcal{B}=(\partial_{x_j}b^i)_{i,j\le N},
$$
$$
L_{a, b}V(x, t)\le (C_0-\Lambda(x, t))V(x), \quad
\Lambda(x, t):=
4\sum_{i, j, k\le N}\bigl|\partial_{x_k}\sigma_N^{ij}(x, t)\bigr|^2
+2\theta(x)+\delta(1+|x|^2)^{-1}|b(x, t)|^2).
$$
Let $s\in (0, T_0)$. Then there exists a number $\kappa>0$ such that for every $M>0$
the Cauchy problem
$$
\partial_tf+\zeta_{M}L_{a, b}f=0,
\quad
f|_{t=s}=\psi,
$$
where $\psi\in C^{\infty}_b(\mathbb{R}^N)$, $\zeta_{M}(x)=\eta\bigl((1+|x|^2)^{\kappa}/M\bigr)$,
has a smooth solution $f$  such that
$$
|f(x, t)|\le \max_x|\psi(x)|, \quad
|\nabla f(x, t)|^2\le e^{(C_0+1)(s-t)}V(x)\max_x|\nabla\psi(x)|^2/2.
$$
\end{lemma}
\begin{proof}
The existence of a smooth bounded (with bounded derivatives) solution $f$
is well known (see \cite[Theorem 2]{Ol}, \cite[Theorem 3.2.4, Theorem 3.2.6]{Str}).
The maximum principle implies that $|f(x, t)|\le \max_x|\psi(x)|$.
Set $u=2^{-1}\sum_{k=1}^N|\partial_{x_k}f|^2$.
Differentiating the equation $\partial_tf+\zeta_ML_{a, b}f=0$ with respect to $x_k$ and
multiplying by $\partial_{x_k}f$, we obtain
\begin{multline*}
\partial_tu+\zeta_ML_{a, b}u+\zeta_{M}\langle \mathcal{B}\nabla f, \nabla f\rangle
+\langle \nabla\zeta_M, \nabla f\rangle \langle b, \nabla f\rangle
+\zeta_{M}\partial_{x_k}a^{ij}\partial^2_{x_ix_j}f\partial_{x_k}f+
\\
+a^{ij}\partial^2_{x_ix_j}f\partial_{x_k}f\partial_{x_k}\zeta_{M}-
\zeta_{M}a^{ij}\partial^2_{x_kx_j}f\partial^2_{x_kx_j}f=0.
\end{multline*}
Note that $\langle \mathcal{B}\nabla f, \nabla f\rangle\le 2\theta u$
and $\langle \nabla\zeta_M, \nabla f\rangle \langle b, \nabla f\rangle
\le 2|\nabla\zeta_M||b|u$.
Let us consider the expression
$$
\zeta_{M}\partial_{x_k}a^{ij}\partial^2_{x_ix_j}f\partial_{x_k}f
+a^{ij}\partial^2_{x_ix_j}f\partial_{x_k}f\partial_{x_k}\zeta_{M}-
\zeta_{M}a^{ij}\partial^2_{x_kx_j}f\partial^2_{x_kx_j}f.
$$
Recall that $A=\sigma_N^2$. We have
\begin{multline*}
\sum_{i, j, k}\partial_{x_k}a^{ij}\partial^2_{x_ix_j}f\partial_{x_k}f=
2\sum_{i, j, m, k}\partial_{x_k}\sigma^{im}\sigma^{mj}\partial^2_{x_ix_j}f\partial_{x_k}f\le
\\
\le2\sum_{i, m}\Bigl(\sum_k|\partial_{x_k}\sigma^{im}|^2\Bigr)^{1/2}
\Bigl(\sum_k|\partial_{x_k}f|^2\Bigr)^{1/2}\Bigl|\sum_j\sigma^{mj}\partial^2_{x_ix_j}f\Bigr|,
\end{multline*}
which is estimated by
$$
4u\sum_{i, m, k}|\partial_{x_k}\sigma^{im}|^2+2^{-1}\sum_{i,m}\Bigl|\sum_j\sigma^{mj}\partial^2_{x_ix_j}f\Bigr|^2.
$$
Note that
$$
\sum_{i,m}\Bigl|\sum_j\sigma^{mj}\partial^2_{x_ix_j}f\Bigr|^2=\sum_{i, j, k}a^{ij}\partial^2_{x_kx_j}f\partial^2_{x_kx_j}f.
$$
Applying the inequality $xy\le (4+4{\rm tr}A)^{-1}x^2+(1+{\rm tr}A)y^2$ we obtain
$$
a^{ij}\partial^2_{x_ix_j}f\partial_{x_k}f\partial_{x_k}\zeta_{M}\le
\frac{|\nabla\zeta_M|^2}{\zeta_M}(1+{\rm tr}A)+
(4+4{\rm tr}A)^{-1}\Bigl(a^{ij}\partial^2_{x_ix_j}f\Bigr)^2.
$$
Note that the following inequality is true:
$$
\Bigl(\sum_{i, j=1}^Na^{ij}\partial_{x_i}\partial_{x_j}f\Bigr)^2\le
\Bigl(\sum_{i=1}^Na^{ii}\Bigr)
\Bigl(\sum_{i, j, k}^Na^{ij}\partial_{x_i}\partial_{x_k}f\partial_{x_j}\partial_{x_k}f\Bigr).
$$
This follows by the inequality
$$
|{\rm tr}\, (AB)|^2\le {\rm tr}\, A \ {\rm tr}\, (AB^2)
$$
valid for symmetric matrices $A$ and $B$, where $A$ is nonnegative. The latter
is due to the Cauchy inequality applied to the inner product
$\langle X,Y\rangle ={\rm tr}\, (XY^{*})$ on the space of $N\times N$-matrices
and the matrices $X=A^{1/2}$, $Y=BA^{1/2}$,  for which
${\rm tr}\, (YY^{*})={\rm tr}\, (BA^{1/2}A^{1/2}B)={\rm tr}\, (AB^2)$.
Applying the above inequality it is easy to verify that
$$
\partial_{t}u+\zeta_ML_{a, b}u+Qu\ge 0,
$$
where
$$
Q=\frac{|\nabla\zeta_{M}|^2}{\zeta_M}(1+{\rm tr} A)+|\nabla\zeta_{M}||b|+
2\zeta_{M}\theta+4\zeta_{M}\sum_{i, j, k\le N}
\bigl|\partial_{x_k}\sigma^{ij}_N\bigr|^2.
$$
We have
$$
|\nabla\zeta_{M}(x)|\le 4\kappa(1+|x|^2)^{-1/2}\bigl|\eta'\bigl((1+|x|^2)^{\kappa}/M\bigr)\bigr|.
$$
Hence
$$
Q\le 4\kappa^2C(1+{\rm tr}\, A)+ 16\kappa C+
\zeta_M\Bigl(4\sum_{i, j, k\le N}
\bigl|\partial_{x_k}\sigma^{ij}_N\bigr|^2+2\theta+2\kappa(1+|x|^2)^{-1}|b|^2\Bigr).
$$
Let us choose $\kappa>0$ such that
$$
Q\le 1+\zeta_M\Bigl(4\sum_{i, j, k\le N}\bigl|\partial_{x_k}\sigma^{ij}_N\bigr|^2+2\theta+\delta(1+|x|^2)^{-1}|b|^2\Bigr).
$$
Let us set $u=wV$. Then $w$ satisfies the inequality
$$
\partial_{t}w+\zeta_{M}L_{a, \widetilde{b}}w+\widetilde{Q}w\ge 0,
$$
where
$$
\widetilde{b}^k=b^k+2\frac{a^{kj}\partial_{x_j}V}{V}, \quad
\widetilde{Q}=Q+\zeta_M\frac{L_{a, b}V}{V}.
$$
By our assumptions we have $\widetilde{Q}\le C_0+1$.
Since $u(x,s)=|\nabla f(x,s)|^2/2=|\nabla\psi(x)|^2/2$, we have
$$
w(x, s)=V(x)^{-1}|\nabla\psi(x)|^2/2\le |\nabla\psi(x)|^2/2.
$$
Applying the maximum principle (see \cite[Theorem 3.1.1]{Str}) we obtain
$$
\max_x|w(x, t)|\le e^{(C_0+1)(s-t)}\max_{x}|\nabla\psi(x)|^2/2,
$$
which completes the proof.
\end{proof}

We can now prove our theorem.

\begin{proof}
Assume that $\sigma^1=\sigma^1_tdt$ and $\sigma^2=\sigma^2_tdt$ belong to $\mathbb{P}_{\nu}$.
By our assumption about $\mathbb{P}_\nu$ we have
$\sigma=(\sigma^1+\sigma^2)/2\in\mathbb{P}_{\nu}$.
Let $d\in\mathbb{N}$, $\psi\in C^{\infty}_0(\mathbb{R}^d)$
and $|\nabla\psi(x)|+|\psi(x)|\le 1$ for all $x\in\mathbb{R}^d$.
For every $\varepsilon>0$ and every natural number $d$ we find a natural number $N\ge d$,
a~$C^{\infty}$-mapping $b=(b^k)_{k=1}^N\colon\,
\mathbb{R}^N\times [0, T_0]\to \mathbb{R}^N$, a function $\theta$ on~$\mathbb{R}^N$,
a function $V\in C^2(\mathbb{R}^N)$, $V\ge 1$,
and numbers $C_0\ge 0$ and $\delta>0$
such that (i) and (ii) in condition~(${\rm B}'$) are fulfilled.

Let a function $\eta\in C^{\infty}_0(\mathbb{R}^1)$ be such that
$\eta(x)=1$ if $|x|\le 1$ and $\eta(x)=0$ if $|x|>2$,
$0\le\eta\le 1$ and there exists a number $C>0$ such that  $|\eta'(x)|^2\eta^{-1}(x)\le C$
for every $x$.
Let $\kappa>0$ be as in Lemma~\ref{lem2.8}.
Set $\varphi_{K}(x)=\eta(|x|^2/K)$ and $\zeta_{M}(x)=\eta\bigl((1+|x|^2)^{\kappa}/M\bigr)$.

Let us fix a number $K>0$ and  find a number $M$ such that $\zeta_{M}(x)=1$ if $|x|^2<2K$.

Fix
$t\in \bigcap_{K} (J^{\sigma^1}_{\psi\varphi_K}\bigcap J^{\sigma^2}_{\psi\varphi_K})$.
Let $f$ be a smooth bounded solution to the finite-dimensional Cauchy problem
$$
\left\{\begin{array}{l}
\partial_tf+\zeta_M(x)\sum_{i, j=1}^{N}a^{ij}\partial_{x_i}\partial_{x_j}f
+\zeta_M(x)\sum_{i=1}^Nb^i\partial_{x_i}f=0 \quad \hbox{on $\mathbb{R}^N\times (0,t)$}, \\
f(t, x)=\psi(x).
\end{array}\right.
$$
Set $\mu=\sigma^1-\sigma^2$.
The measure $\mu$ solves the Cauchy problem (\ref{e1})
with zero initial condition.
Recall that $\zeta_M(x)=1$ if $\varphi_K(x)\ne 0$.
Therefore,
\begin{equation*}
\int_{\mathbb{R}^{\infty}}\psi\varphi_K
\,d\mu_t=\int_0^t\int_{\mathbb{R}^{\infty}}
\bigl[\varphi_K \langle B-b, \nabla_x f\rangle
+fL\varphi_K+2\langle A\nabla_xf, \nabla_{x}\varphi_K\rangle\bigr]\,d\mu_s\,ds.
\end{equation*}
Applying Lemma \ref{lem2.8} we have the estimate
$$
|f(x, s)|\le 1, \quad |\nabla_x f(x, s)|^2\le e^{(C_0+1)(T_0-s)}V(x)/2.
$$
Hence
\begin{equation*}
\int_{\mathbb{R}^{\infty}}\psi
\,d\mu_t\le 2\int_0^t\int_{\mathbb{R}^{\infty}}
\Bigl[|B-b|V^{1/2}e^{(C_0+1)(T_0-s)/2}+|L\varphi_K|+2|A\nabla\varphi_K|e^{(C_0+1)(T_0-s)/2}V^{1/2}\Bigr]\,d\sigma_s\,ds.
\end{equation*}
Letting $K\to +\infty$ we find that
$$
\int_{\mathbb{R}^{\infty}}\psi
\,d\mu_t\le 2\int_0^t\int_{\mathbb{R}^{\infty}}
|B-b|V^{1/2}e^{(C_0+1)(T_0-s)/2}\,d\sigma_s\,ds<2\varepsilon.
$$
Since $\varepsilon>0$ was arbitrary, we obtain
$$
\int_{\mathbb{R}^{\infty}}\psi\,d\mu_t\le 0.
$$
Replacing $\psi$ by $-\psi$ we arrive at the equality
$$
\int_{\mathbb{R}^{\infty}}\psi\,d\mu_t=0.
$$
Therefore,
$$
\int_{\mathbb{R}^{\infty}}
\psi\,d\sigma^1_t=\int_{\mathbb{R}^{\infty}}\psi\,d\sigma^2_t
$$
for almost every $t$. Thus, $\sigma^1=\sigma^2$.
\end{proof}

\begin{example}\label{ex2.9}
\rm (''Reaction diffusion equations'')
Let us return to the situation of Example~\ref{ex2.4}, but
now we assume that
there exists a sequence of smooth bounded functions $f_n(z, u,t)$ such that
$\lim\limits_{n\to\infty}f_n(z, u,t)=f(z, u,t)$ for every $u, t, z$ and
 $$
 |f_n(z, u,t)|\le C_1+C_1|u|^m,
 \quad
(f_n(z, u,t)-f_n(z, v,t))(u-v)\le C_2|u-v|^2,
$$
where $C_1$ and $C_2$ do not depend on $n$.
Assume also that $a^{ij}=(Se_i, e_j)_{2}$ for some symmetric nonnegative operator $S$
on $L^2((0, 1))$, which can be degenerate unlike in Example~\ref{ex2.4}.
Then there exists at most one probability solution $\mu$ of the Cauchy problem for
the Fokker--Planck--Kolmogorov equation $\partial_t\mu=L^{*}\mu$
such that
$$
\int_0^{T_0}\int_{L^2((0, 1))}\|u\|_{2m}^{m}\,\mu_t(du)\,dt<\infty.
$$
The same conclusion is true if  $A=(a^{ij})$ is a
nonconstant matrix satisfying
condition (${\rm A}')$ and there exists a constant $C_1$ such that
for every natural number $N$ and every $(x, t)\in\mathbb{R}^N\times[0, T_0]$
we have
$$
\sum_{i, j, k\le N}\bigl|\partial_{x_k}\sigma^{ij}_N(x, t)\bigr|^2\le C_1.
$$
\end{example}
\begin{proof}
Set $F^i(u,t)=\langle f(\,\cdot\,,u(\cdot),t), e_i\rangle_{2}$,
$F^i_n(u,t)=\langle f_n(\,\cdot\,,u(\cdot),t), e_i\rangle_{2}$,
$F_n(u,t)=(F^i_n(u,t))_{i=1}^\infty$,
and extend all these maps to all of $\mathbb{R}^\infty\times [0,T_0]$ by zero.
According to our assumptions and the dominated convergence theorem we have
$$
\lim_{n\to\infty}\int_0^{T_0}\int_{L^2((0, 1))}\|F(u,t)-F_n(u,t)\|_{l^2}\,\mu_t(du)\,dt=0.
$$
Let $P_Nu:=u_1e_1+\ldots+u_Ne_N$.
The above equality shows
that for each $\varepsilon>$ and $d\ge 1$ there exist numbers $n$ and $N>d$ such that
$$
\int_0^{T_0}\int_{L^2((0, 1))}\|F(u,t)-F_n(P_Nu,t)\|_{l^2}\,\mu_t(du)\,dt<\varepsilon.
$$
Note that the condition
$$
(f_n(z, u,t)-f_n(z, v,t))(u-v)\le C_2|u-v|^2
$$
implies that
$$
\sum_{i,j\le N}
\partial_{u_i}F^j_n(P_Nu,t)h_ih_j\le C_2|h|^2, \quad h=(h_i)\in \mathbb{R}^N.
$$
Hence Theorem \ref{th2} with $V\equiv 1$ implies uniqueness.
\end{proof}

Below for simplicity the integral of the product of an integrable function
$f_1$ and a bounded function $f_2$ is denoted by $(f_1,f_2)_2$.

\begin{example}\label{ex2.10}
\rm (``Stochastic Burgers equation'')
Suppose that $\{e_k\}$
is an eigenbasis of  the Laplacian
on $L^2[0, 1]$ with zero boundary condition, i.e., $D^2 e_k=-\lambda_k^2e_k$.
Set $B(u)(z)=D^2u(z)+D(u^2(z))$, that is,
$$
B^i(u)=-\lambda_i^2u_i-\langle u^2, De_i\rangle_{2}, \ u\in L^2[0, 1], \ u_i=\langle u,e_i\rangle_{2}.
$$
Assume that $a^{ij}=\langle Se_i, e_j\rangle_{2}$ for some symmetric nonnegative operator $S$
on $L^2[0, 1]$ with
finite trace (${\rm tr} S<\infty$).
Set
$$
L\varphi=\sum_{i, j=1}^{\infty}a^{i j}\partial_{e_i}\partial_{e_j}\varphi
+\sum_{i=1}^{\infty}B^i\partial_{e_i}\varphi.
$$
Let $H_0^1$ be the space of all absolutely continuous functions $u$ on $[0,1]$
such that $u(0)=u(1)=0$ and $\|u\|_{H_0^1}:=\|u'\|_2<\infty$.
Then there exists at most one probability solution $\mu$ of the Cauchy problem for
the Fokker--Planck--Kolmogorov equation $\partial_t\mu=L^{*}\mu$ such that
$$
\int_0^{T_0}\int_{L^2[0, 1]}\|u\|_{H_0^1}^2e^{\delta\|u\|_{2}^2}\,\mu_t(du)\,dt<\infty
$$
for some $\delta>0$.
\end{example}
\begin{proof}
We apply Example~\ref{ex2.7}(ii). Recall that the matrix $(a^{ij})$ has to satisfy the following condition
for some $\varepsilon_0>0$:
$$
\varepsilon_0\bigl(\langle A_Nx, x\rangle +|x|^2\bigr)\le \|x\|_{l^2_{\lambda}}^2,
\quad x\in\mathbb{R}^N.
$$
This is equivalent to
$$
\varepsilon_0\bigl(\langle Su, u\rangle_2+\|u\|^2_2\bigr)\le \|u\|_{H_0^1}^2,
$$
which is true for sufficiently small $\varepsilon_0$. We fix $\varepsilon_0\in (0, \delta)$.
Set $F^i(u):=\langle u^2, De_i\rangle_{2}$ for $u\in L^2$ and extend $F^i$ by zero
to all other $u=(u_k)$ in $\mathbb{R}^\infty$.
Let
$F(u)=(F^i(u))_{i=1}^\infty$,
$P_Nu:=u_1e_1+\ldots+u_Ne_N$,
$$
b^k(u_1, \ldots, u_N):=-\lambda_k^2u_k+F^k(P_Nu), \quad k\le N.
$$
Note that
$$
\|F(u)\|_{l^2}=\|(u^2)'\|_2=2\|uu'\|_2\le 2\|u\|_{H_0^1}^2.
$$
Hence
$$
\lim_{N\to\infty}\int_0^{T_0}\int_{L^2[0, 1]}
\|F(u)-F(P_Nu)\|_{l^2}e^{\delta\|u\|_{2}^2}\,\mu_t(du)\,dt=0
$$
It is easy to see that $|b^k(u)|\le C_1(N)+C_2(N)\|P_Nu\|_2^2$ and
$\langle F(P_Nu), P_Nu\rangle_{2}\le 0$.
Moreover, for every $\gamma\in(0, 1)$ we have the inequalities
$$
\sum_{i,k\le N} \partial_{u_i}F^k(P_Nu)h_ih_k
\le \|h\|_{l^2_{\lambda}}+(\gamma\|P_Nu\|_{H_0^1}^{2}+C_{\gamma})|h|^2,
\quad h=(h_i)\in\mathbb{R}^N.
$$
Set
$\theta(P_Nu)=\gamma\|P_Nu\|_{H_0^1}^{2}+C_{\gamma}$ and $C_0=C_{\gamma}+{\rm tr} S$
(we recall that ${\rm tr} S<\infty$). In order to apply
Example~\ref{ex2.7}(ii) we choose $\gamma<2^{-1}\delta(\varepsilon_0-\delta)$.
\end{proof}

\begin{example}\label{ex2.11}
\rm (``Mixed Burgers/reaction diffusion type equations'')
(i)
In the situation of the previous example we consider the operator $L$ with
the drift coefficient of the form
$$
B(u)(z)=D^2u(z)+D(u^2(z))-u^{2m+1}(z), \quad m\in\mathbb{N},
$$
that is,
$$
B^i(u)=-\lambda_i^2u_i-\langle u^2, De_i\rangle_{2}-\langle u^{2m+1}, e_i\rangle_{2}.
$$
Assume that $a^{ij}$ satisfies the assumptions in the previous example.
Then there exists at most one probability solution $\mu$ of the Cauchy problem for
the Fokker--Planck--Kolmogorov equation $\partial_t\mu=L^{*}\mu$ such that
$$
\int_0^{T_0}\int_{L^2[0, 1]}\bigl[\|u\|_{4m+2}^{2m+1}
+\|u\|_{H_0^1}^2\bigr]e^{\delta\|u\|_{2}^2}\,\mu_t(du)\,dt<\infty
$$
for some $\delta>0$.

(ii)
In the situation of Example~\ref{ex2.10} we consider the operator $L$ with
the drift coefficient of the form
$$
B(u)(z)=D^2u(z)+D(u^{m}(z))-u^{2l+1}(z), \quad 2\le m\le l+2, \quad m, l\in\mathbb{N}
$$
that is,
$$
B^i(u)=-\lambda_i^2u_i-\langle u^{m}, De_i\rangle_{2}-\langle u^{2l+1}, e_i\rangle_{2}.
$$
Assume also that $a^{ij}=0$ if $i\ne j$ and that
$
\sum_{i=1}^\infty a^{ii}<\infty.
$
Then there exists at most one probability solution $\mu$ of the Cauchy problem for
the Fokker--Planck--Kolmogorov equation $\partial_t\mu=L^{*}\mu$ such that
$$
\int_0^{T_0}\int_{L^2((0, 1))}
\bigl[\|u\|_{4l+2}^{2l+1}+\|u^{m}\|_{H_0^1}\bigr]
\exp \Bigl(\kappa'\|u\|_{2m-2}^{2m-2}\Bigr)\,\mu_t(du)\,dt<\infty
$$
for some $\kappa'>0$.
This partially improves the results in \cite{GR}.
\end{example}
\begin{proof}
(i) We apply Example~\ref{ex2.7}(ii). Note that as in the above example
the matrix $(a^{ij})$ satisfies all conditions in Example~\ref{ex2.7}(ii).
Let $\psi_M\in C^{\infty}(\mathbb{R}^1)$, $\psi(s)=-\psi(-s)$,
$0\le \psi'\le 1$, $\psi_M(s)=s$ if $|s|\le M-1$ and
$\psi_M(s)=M$ if $s>M+1$.
Set
$$
F^i(u):=-\langle u^2, De_i\rangle_{2}-\langle u^{2m+1}, e_i\rangle_{2}, \quad
F^i_M(u):=-\langle u^2, De_i\rangle_{2}-\langle \psi_M(u)^{2m+1}, e_i\rangle_{2},
$$
$$
P_Nu:=u_1e_1+\ldots+u_Ne_N,
\quad
b^k(u_1, \ldots, u_N):=-\lambda_k^2u_k+F_M^k(P_Nu).
$$
As above, we define all these functions by zero if $u$ is not in $L^2[0,1]$.
Note that $\|F(u)\|_{l^2}\le 2\|u\|_{H_0^1}^2+\|u\|_{4m+2}^{2m+1}$.
and the same is true for $F_M(u)$ in place of $F(u)$.
Hence
$$
\lim\limits_{N\to\infty}
\Bigl(\lim\limits_{M\to\infty}
\int_0^{T_0}\int_{L^2[0, 1]}
\|F(u)-F_M(P_Nu)\|_{l^2}e^{\delta\|u\|_{2}^2}\,\mu_t(du)\,dt\Bigr)=0.
$$
It is easy to see that $|b^k(u)|\le C_1(N)+C_2(N)\|P_Nu\|_2^2$.
Recall that $\psi_M'\ge 0$ and $\psi_M(s)=-\psi_M(-s)$. Hence
$\langle F_M(P_Nu), P_Nu\rangle_{2}\le 0$.
For every $\gamma\in(0, 1)$ we have
$$
\sum_{i,k\le N}
\partial_{u_i}F^k_M(P_Nu)h_ih_k\le \|h\|_{l^2_{\lambda}}+(\gamma\|P_Nu\|_{H_0^1}^{2}+C_{\gamma})|h|^2,
\quad h=(h_i)\in \mathbb{R}^N.
$$
Set $\theta(P_Nu)=\gamma\|P_Nu\|_{H_0^1}^{2}+C_{\gamma}$ and $C_0=C_{\gamma}+{\rm tr} S$
(we recall that ${\rm tr} S<\infty$). In order to apply
Example~\ref{ex2.7}(ii) we choose $\gamma<2^{-1}\delta(\varepsilon_0-\delta)$.

(ii)
We check the condition (${\rm B'}$).
Set $F^k(u)=-\langle u^m, De_k \rangle_2-\langle u^{2l+1}, e_k \rangle_2$,
$P_Nu=u_1e_1+\ldots+u_Ne_N$ and
$$
b^k(u_1, u_2, \ldots, u_N)=-\lambda_k^2u_k+F^k(P_Nu).
$$
For each $N\ge 1$ and $p\ge 1$ there exist positive numbers $C_1(N, p)$ and $C_2(N, p)$
such that
$$
C_1\|P_Nu\|_2^{2p}\le \|P_Nu\|_{L^p}^p\le C_2\|P_Nu\|_2^{2p}.
$$
Hence there exists a number $C_3(N)$ such that
$$
|b(P_Nu)|(1+\|P_Nu\|_2^2)^{-1}\le
C_3(N)+C_3(N)\Bigl(\int_0^1|P_Nu(z)|^{m-2}\,dz+\int_0^1|P_Nu(z)|^{2l-1}\,dz\Bigr).
$$
It is easy to see that for every $\gamma\in(0, 1)$ there exists a number $C_{\gamma}>0$
(independent of~$N$) such that
$$
\sum_{i,k\le N}
\partial_{u_i}F^k(u^N)h_ih_k\le \gamma\|h\|_{l^2_{\lambda}}
+(\gamma\|(P_Nu)^{m-1}\|_{H_0^1}^{2}+C_{\gamma})|h|^2,
\quad h=(h_i)\in\mathbb{R}^N.
$$
Set $V(u)=\exp(\kappa\|u\|_{2m-2}^{2m-2})$, where $0<\kappa<\kappa'$
(the number $\kappa'$ comes from our assumptions).
Using the inequalities $\sum_{i=1}^{\infty}a^{ii}u_i^2\le c_0\sum_{i=1}^{\infty}$ for some $c_0>0$
and all $u\in L^2([0, 1])$, $\sum_{i=1}^{\infty}a^{ii}<\infty$, $m\le l+2$ and
choosing a sufficiently small number $\kappa$, we obtain
$$
L_{a, b}V(P_Nu)=\Bigl(C(m)-2^{-1}\kappa\|(P_Nu)^{m-1}\|_{H_0^1}^{2}
-2^{-1}\kappa\|P_Nu\|_{2m+2l-2}^{2m+2l-2}\Bigr)V(P_Nu).
$$
for some number $C(m)>0$ which does not depend on $N$.
Note that $m-2\le 2m+2l-2$ and $2l-1\le 2m+2l-2$. Choosing $\gamma<\kappa/2$ we have
$$
L_{a, b}V(P_Nu)\le (C_0-\gamma\|(P_Nu)^{m-1}\|_{H_0^1}^{2}-\delta|b(P_Nu)|(1+\|P_Nu\|_{2}^2)^{-1})V(P_Nu)
$$
for some $C_0>0$ and $\delta>0$. Note that $C_0$ does not depend on $N$ and we can omit the term
$e^{C_0(T-t)/2}$ in the condition (${\rm B'}$)(i).
Finally we note that $\|F(u)\|_{l^2}\le \|u^{m}\|_{H_0^1}+\|u\|_{4l+2}^{2l+1}$ and
$$
\lim_{N\to\infty}\int_0^{T_0}\int_{L^2((0, 1))}
\|F(u)-F(P_Nu)\|_{l^2}\exp\Bigl(\kappa\|u\|_{2m-2}^{2m-2}\Bigr)\,\mu_t(du)\,dt=0.
$$
Hence Theorem \ref{th2} implies uniqueness.
\end{proof}

\begin{example}\label{ex2.12}
\rm (``Stochastic $2d$-Navier--Stokes equation'')
Let us consider the space
$V_2$ of $\mathbb{R}^2$-valued mappings $u=(u^1, u^2)$
such that $u^j\in H_0^{2,1}(D)$ and ${\rm div}\, u=0$,
where $D\subset\mathbb{R}^2$ is a bounded domain with smooth
boundary. The space $V_2$ is equipped with its natural Hilbert norm
$\|u\|_{V_2}$ defined by
$$
\|u\|_{V_2}^2:=\sum_{j=1}^2 \|\nabla_z u^j\|_{2}^2.
$$
Let $H$ be the closure of $V_2$ in $L^2(D,\mathbb{R}^2)$
and let $P_H$ denote the orthogonal projector on $H$ in $L^2(D,\mathbb{R}^2)$.
It is known (see \cite{Lad}) that there exists
an orthonormal basis $\{\eta_n\}$ in~$H$ formed
by eigenfunctions of $\Delta$ with eigenvalues $-\lambda_n^2$ such that $\eta_n\in V_2$.
Recall that $\langle P_H w,\eta_n\rangle_2=\langle w,\eta_n\rangle_2$ for any
$w\in L^2(D,\mathbb{R}^d)$.
Set
$$
B^n(u,t) =
\langle u,\Delta \eta_n\rangle_2
-\sum_{j=1}^2 \langle P_H u^j\partial_{z_j} u, \eta_n\rangle_{2}=
\langle u,\Delta \eta_n\rangle_2
-\sum_{j=1}^2 \langle\partial_{z_j}u, u^j\eta_n\rangle_{2}
$$
whenever $u\in V_2$ and $B^n(u,t) =0$ otherwise.
These functions are continuous
on balls in~$V_2$ with respect to the topology of $L^2(D,\mathbb{R}^2)$,
which easily follows from the compactness of the Sobolev embedding $H^{2,1}(D)\to L^2(D)$.
Consider the operator
$$
L\varphi(u,t)=
\sum_{i, j}^\infty a^{i j} \partial_{\eta_i}\partial_{\eta_j} \varphi (u,t)
+\sum_{n=1}^\infty B^n(u,t)\partial_{\eta_n}\varphi (u,t).
$$
Assume that $a^{ij}=\langle S\eta_i, \eta_j\rangle_{2}$
for some symmetric nonnegative bounded operator $S$ on~$H$.
Suppose also that $\sum_ia^{ii}\lambda_i^2<\infty$.
Then there exists at most one probability solution $\mu$ of the Cauchy problem
for the Fokker--Planck--Kolmogorov equation $\partial_t\mu=L^{*}\mu$ such that
for some $\delta>0$
$$
\int_0^{T_0}\int_{H}\bigl(1+\|\Delta u\|_{2}^2\bigr)
e^{\delta\|u\|_{V_2}^2}\,\mu_t(du)\,dt<\infty,
$$
where we set $\|\Delta u\|_{2}=\infty$ if $u\not\in H^{2,2}(D)$.
\end{example}
\begin{proof}
We apply Example~\ref{ex2.7}(iii). Recall that
the matrix $(a^{ij})$ has to satisfy the following condition
for some $\varepsilon_0>0$
$$
\varepsilon_0\sum_{i,j\le N}a^{ij}\lambda^2_i\lambda^2_jx_ix_j+\varepsilon_0|x|^2\le \sum_{i\le N}
\lambda_i^4x_i^2
$$
that is equvivalent (if we take $u=\sum_{i=1}^N\lambda_ix_ie_i$)
to the estimate
$$
\varepsilon_0\bigl(\langle Su, u\rangle_2+\|u\|_2^2\bigr)\le \|u\|^2_2,
$$
which is true for sufficiently small $\varepsilon_0$.
Set
$$
F^n(u)=-\sum_{j=1}^2 \langle \partial_{z_j}u, u^j\eta_n\rangle_{2},
\quad u\in V_2.
$$
Note that $|F^n(u)|\le C_1(n)+C_2(n)\|u\|^2_{2}$, since
$F^n(u)=\sum_{j=1,2} \langle u,u_j\partial_{z_j}\eta_n\rangle_2$
due to the condinition that ${\rm div}\, u =0$.
It is well-known that there exists a constant $C_1>0$ such that
for every function $g\in H^{2, 1}_0(D)\cap H^{2, 2}(D)$ we have
$$
\|g\|_{2, 2}\le C_1\bigl(\|\Delta g\|_{2}+\|g\|_{2}\bigr).
$$
Moreover, for every $g\in H^{2, 2}(D)$, every $r\ge 1$ and some constant $C_2>0$
we have
$$
\|g\|_{r}\le C_2\|g\|_{2, 1}.
$$
Hence
\begin{align*}
\|F(u)\|_{l^2}^2 &\le \int_{D}|\nabla_z u(z)|^2|u(z)|^2\,dz\le
\Bigl(\int_{D}|\nabla_z u(z)|^4\,dz\Bigr)^{1/2}\Bigl(\int_{D}|u(z)|^4\,dz\Bigr)^{1/2}\le
\\
&\le
C_1^2C_2^4(1+\|\Delta u\|_{2}^2)\|u\|_{V_2}^2.
\end{align*}
Let $P_Nu=u_1\eta_1+\ldots+u_N\eta_N$.
We have
$$
\lim_{N\to\infty}\int_0^{T_0}\int_{L^2((0, 1))}\|F(u)-F(P_Nu)\|_{l^2}e^{\delta\|u\|_{V_2}^2/2}\,\mu_t(du)\,dt=0.
$$
It is known (see, e.g., \cite[Proposition 6.3]{D}) that in the considered case $d=2$ we have
the inequality
$$
\langle F(P_Nu), \Delta P_Nu\rangle_{2}=0
$$
which gives the condition
$\langle f(x, t), x\rangle_{l^2_{\lambda}}\le 0$ in Example~\ref{ex2.7}(iii).
In addition, for every $\gamma\in(0, 1)$
$$
\sum_{i,j\le N} \partial_{u_i}F^j(P_Nu)h_ih_j\le \bigl(C_\gamma
+\gamma\|\Delta P_Nu\|_{l^2}^2\bigr)|h|^2+\|h\|^2_{l^2_{\lambda}},
\quad h=(h_i).
$$
Set $\theta(P_Nu)=C_\gamma+\gamma\|\Delta P_Nu\|_{l^2}^2$
and $C_0=C_{\gamma}+\sum_{i=1}^{\infty}a^{ii}\lambda_i^2$
(we recall that $\sum_{i=1}^{\infty}a^{ii}\lambda_i^2<\infty$). In order to apply
Example~\ref{ex2.7}(iii) we choose $\gamma<2^{-1}\delta(\varepsilon_0-\delta)$.
In Example~\ref{ex3.5} we consider a more general equation.
\end{proof}

It is worth noting that the last example applies to degenerate coefficients $A$, in particular,
to $A$ identically zero, which gives uniqueness for the so-called continuity equation
corresponding to $2d$-Navier--Stokes equation.

In the next section we show that the considered classes of uniqueness are not empty.

\section{Existence of solutions}

First we would like to mention that if the stochastic equation associated
to our Fokker--Planck--Kolmogorov equation has a solution in the sense
of Stroock--Varadhan's martingale problem, then one immediately gets
a solution to the FPK-equation. But uniqueness of solutions for a martingale problem
does not imply uniqueness for the corresponding FPK-equation.

In this section we purely analytically prove
the following existence result generalizing a result from \cite{BDPR08}
(where only a sketch of the proof of a weaker result was given).

Let $\{e_n\}$ be an orthonormal basis in $l^2$.
The
linear span of $e_1,\ldots,e_n$ is denoted by~$H_n$.

Let $T_0>0$ and let $a^{ij}\colon\, \mathbb{R}^\infty\times [0,T_0]\to \mathbb{R}^1$ and
$B^i\colon\,  \mathbb{R}^\infty\times [0,T_0]\to \mathbb{R}^1$ be Borel
functions. Suppose that the matrices $(a^{ij})_{i,j\le n}$ are
symmetric nonnegative for all~$n$. Set
$$
L\varphi(x,t):=\sum_{i,j=1}^n
a^{ij}(x,t)\partial_{e_i}\partial_{e_j}\varphi(x,t)+ \sum_{i=1}^n
B^{i}(x,t)\partial_{e_i}\varphi(x,t), \ (x,t)\in \mathbb{R}^\infty \times [0,T_0]
$$
for functions $\varphi$ that are smooth functions of the variables
$x_1,\ldots,x_n,t$.

Let $B_n:=(B^1,\ldots,B^n)$ and $P_nx=(x_1,\ldots,x_n)$.

A Borel function  $\Theta\colon\,  \mathbb{R}^\infty\to [0,+\infty]$ such that
the sublevel sets $\{\Theta\le R\}$ are compact is called a compact function.
For example, one can take any numbers $\alpha_i>0$  and set
$\Theta(x)=\sum_{i=1}^\infty \alpha_i^2x_i^2.$

\begin{theorem}\label{t4.1}
Suppose that there exists
a compact function $\Theta\colon\, \mathbb{R}^\infty\to [0,+\infty]$,
finite on each $H_n$ and such that
 the functions $a^{ij}$ and $B^i$ are continuous in~$x$ on all the sets
$\{\Theta\le R\}$,
and there exist   numbers $M_0,C_0\ge 0$ and a Borel
function $V\colon\, \mathbb{R}^\infty\to [1,+\infty]$ whose
 sublevel sets $\{V\le R\}$ are compact and whose restrictions to $H_n$ are of class $C^2$
  and  such that for all $x\in H_n$, $n\ge 1$, one has
\begin{equation}\label{e4.1}
\sum_{i,j=1}^n a^{ij}(x,t)\partial_{e_i}V(x)\partial_{e_j}V(x)\le M_0 V(x)^2,
\quad
LV(x,t)\le C_0 V(x)-\Theta(x) .
\end{equation}
Assume also that there exist constants $C_i\ge 0$ and $k_i\ge 0$
such that for all $i$ and $j\le i$ one has
\begin{equation}\label{e4.2}
|a^{ij}(x,t)|+|B^i(x,t)|\le C_i V(x)^{k_i}(1+\delta(\Theta(x))\Theta(x)),
\ (x,t)\in \mathbb{R}^\infty \times [0,T_0],
\end{equation}
where $\delta$ is a bounded nonnegative Borel function on $[0,+\infty)$
with $\lim\limits_{s\to \infty}\delta(s)=0$.
Then, for every Borel probability measure $\nu$ on $\mathbb{R}^\infty$ such that
$W_k:=\sup_n \|V^k\circ P_n\|_{L^1(\nu)}<\infty$ for all $k\in \mathbb{N}$,
the Cauchy problem {\rm(\ref{e1})}  with initial distribution $\nu$ has a
solution of the form $\mu=\mu_t\, dt$ with Borel probability measures
$\mu_t$ on~$\mathbb{R}^\infty$ such that for all $t\in [0,T_0]$
\begin{equation}\label{e4.3}
\int_{\mathbb{R}^\infty} V^k\, d\mu_t
+k \int_0^t \int_{\mathbb{R}^\infty} V^{k-1} \Theta\, d\mu_s\, ds
\le N_k W_k
\quad \forall \, k\in\mathbb{N},
\end{equation}
where
$N_k:=M_ke^{M_k}+1, \quad M_k=k(C_0+(k-1)M_0)$.
In particular, $\mu_t(V<\infty)=1$ for all $t$ and $\mu_t(\Theta<\infty)=1$ for almost all $t$.
\end{theorem}
\begin{proof}
For every fixed $n$ let $a_n^{ij}$ denote the restriction of
$a^{ij}$ to $H_n\times (0,T_0)$ and set
$A_n:=(a^{ij}_n)_{i,j\le n}$. Denote by $\nu_{n}$ the
projection of $\nu$ on~$H_n$. We show that there exist Borel
probability measures $\mu_{t,n}$ on $H_n$ such that the measure
$\mu_n:=\mu_{t,n}\, dt$ solves the Cauchy problem with
coefficients $A_n$ and $B_n$ on $H_n\times (0,T_0)$ and initial
distribution~$\nu_{n}$. To this end we consider the Lyapunov
function $V_m(x)=V(x)^{m}$ on $H_n$, where $m\ge 1$.
Letting  $M_m:=m(C_0+(m-1)M_0)$,
 we obtain
\begin{align*}
LV_m &=mV^{m-1}\Bigl(LV+(m-1)V^{-1} \sum_{i,j=1}^n a^{ii}\partial_{e_i}V\partial_{e_j}V\Bigr)
\le
mV^{m-1}(C_0V-\Theta+ (m-1)M_0V)
\\
&
\le
 M_mV^{m}-m V^{m-1}\Theta .
\end{align*}
Since the function $V_m$ is $\nu_{n}$-integrable, we can apply
the existence result from \cite{BDR08} and obtain the
desired probability measures $\mu_{t,n}$ on $H_n$ such that
the function
$$
t\mapsto \int_{H_n} \zeta(x)\, \mu_{t,n}(dx)
$$
is continuous on $[0,T_0)$ for every $\zeta\in C_0^\infty(H_n)$.
 Moreover, by \cite[Lemma~1]{BDPR08} (see also \cite[Lemma~2.2]{BDR08}), for each $m\ge 1$ and
 $$
 N_m:=M_me^{M_m}+1, \quad M_m=m(C_0+(m-1)M_0)
 $$
 the following estimate
holds for almost all $t\in (0,T_0)$:
\begin{multline}\label{e4.4}
\int_{H_n} V_m(x)\, \mu_{t,n}(dx)
+ m\int_0^t \int_{H_n} V_{m-1}(x)\Theta(x)\, \mu_{s,n}(dx)\, ds
\\
\le N_m\int_{H_n} V_m(x)\, \nu_{n}(dx)\le N_m+N_mW_m.
\end{multline}
 Therefore, by Fatou's theorem and the
above stated continuity of $t\mapsto \mu_{t,n}$ it follows that
(\ref{e4.4}) holds for all $t\in [0,T_0)$. Indeed, we replace $V_m$ and $\Theta V_{m-1}$
in the left-hand side by $\min(k,V_m)$ and
$\min(k,\Theta V_{m-1})$, obtain the desired estimate for all $t\in [0,T_0)$ keeping $k$
fixed and then let $k\to\infty$.

Suppose now that $\zeta\in C_0^\infty (\mathbb{R}^d)$. Let us identify $H_n$
with $\mathbb{R}^n$.
If $n\ge d$, then $\zeta$ regarded as a function on $\mathbb{R}^n$ belongs
to the class $C_b^\infty(\mathbb{R}^n)$.
Let $m=\max(k_1,\ldots,k_d)$.
Then we have the estimate
\begin{equation}\label{e4.5}
|L\zeta(x,t)|
\le K+KV_m(x)+K V_m(x)
\delta(\Theta(x))\Theta(x), \ (x,t)\in \mathbb{R}^n\times [0,T_0],
\end{equation}
where $K$ is some number which depends on~$\zeta$
(but is independent of $n$ since $\zeta$ is a function of $x_1,\ldots,x_d$).
Therefore, by approximation, inequality (\ref{e4.4}) and Lebesgue's dominated convergence
theorem we have
\begin{equation}\label{e4.6}
\int_{H_n} \zeta(x)\, \mu_{t,n}(dx)
=\int_0^t \int_{H_n} L\zeta(x,s)\, \mu_{s,n}(dx)\, ds
+\int_{H_n}\zeta(x)\, \nu_{n}(dx),
\end{equation}
because according to \cite{BDR08}, this identity holds for all
$\zeta\in C_0^\infty(\mathbb{R}^n)$, hence in our situation it remains valid also
for all~$\zeta\in C_b^\infty(\mathbb{R}^n)$.
Letting
$$
\varphi_n(t):=\int_{H_n} \zeta(x)\, \mu_{t,n}(dx), \ t\in [0,T_0],
$$
we see from (\ref{e4.4}), (\ref{e4.6})
that the function $\varphi_n$ is Lipschitzian (one can show that it
is everywhere differentiable in $(0,T_0)$) and (\ref{e4.5}) yields that
$$
|\varphi_n'(t)|
\le
\int_{H_n} |L\zeta(x,t)|\, \mu_{t,n}(dx)\le
K_\zeta \int_{H_n} [1+V_{m-1}(x)\Theta(x)]\, \mu_{t,n}(dx)
$$
with some number $K_\zeta$ that does not depend on~$n$ (but only on $\zeta$).
Therefore, by (\ref{e4.4})
the functions $\varphi_n$ possess uniformly bounded variations,
hence there is a subsequence in $\{\varphi_n\}$ convergent pointwise on $[0,T_0]$.
We may assume that this is true for the whole sequence.
Moreover, we can do this in a such a way that this pointwise convergence
holds for every function $\zeta$ from a fixed countable family $\mathcal F$
with the following property: the weak convergence of a uniformly
tight sequence of probability measures on $\mathbb{R}^\infty$ follows from convergence
of their integrals of every function in $\mathcal F$.

It follows from (\ref{e4.4}) and the compactness of the sets $\{V_m\le R\}$
  and  $\{\Theta\le R\}$ that,
for every fixed $t\in (0,T_0)$,
the sequence of measures $\mu_{t,n}$ is uniformly tight on $\mathbb{R}^\infty$
(see \cite[Example~8.6.5]{mera}).
  Hence we can find a subsequence, denoted
for simplicity by the same indices~$n$, such that $\{\mu_{t,n}\}$ converges weakly
on $\mathbb{R}^\infty$ for every rational~$t\in (0,T_0)$.
However, since we have ensured
convergence  of $\varphi_n(t)$ at every $t\in [0,T_0]$ for every $\zeta\in \mathcal F$,
we see  that $\{\mu_{t,n}\}$ converges weakly  for every~$t\in [0,T_0]$.

Estimate (\ref{e4.3}) follows from (\ref{e4.4}) taking into account
that $V\ge 1$ and $\Theta\ge 0$ are lower semicontinuous, hence
$V^k$ and $V^{k-1}\Theta$ are lower continuous.

 The family of measures $\mu_t$ obtained in this way
is the desired solution. Indeed, let us fix $\zeta\in C_0^\infty
(\mathbb{R}^d)$. We have to show that the integrals of
$L\zeta(x,t)$ over $\mathbb{R}^\infty\times [0,T]$, $T<T_0$, with
respect to $\mu_{n}$ converge to the integral with respect
to~$\mu=\mu_t\, dt$. This amounts to establishing such convergence
for all functions $f=\partial_{x_i}\zeta B^i$ and
$f=a^{ij}\partial_{x_j}\partial_{x_i}\zeta$. Suppose we are able
to show this for the functions $f_N=\max(\min(f,N),-N)$. Then
(\ref{e4.2}) and (\ref{e4.4}) enable us to extend the same to the
original function~$f$, because for every $\varepsilon>0$ these
estimates give a number $N$ such that the integral of $|f|
I_{|f|>N}$ with respect to $\mu_{t,n}\, dt$ is less
than~$\varepsilon$. Indeed, it suffices to show that the integral
of $G:=V^{k}(1+\delta(\Theta)\Theta)$ over the set $\{G\ge N\}$
with respect to $\mu_{t,n}\, dt$ does not exceed $\varepsilon$ for $N$ sufficiently large.
Take $n_1$ such that
$1/n_1+\delta(s)< c\varepsilon$ for all $s\ge n_1$,
where $c>0$ is so small that $cN_{k+1}W_{k+1}< 1/2$.
We may assume that $\delta\le 1$.
We have
$$
\int_0^{T_0}\int_{\{\Theta\ge n_1\}} G\, d\mu_{t,n}\, dt
=  \int_0^{T_0}\int_{\{\Theta\ge n_1\}} (\Theta^{-1}+\delta(\Theta))
V^{k}\Theta\, d\mu_{t,n}\, dt
\le
c\varepsilon \int_0^{T_0}\int_{H_n} V^{k}\Theta\, d\mu_{t,n}\, dt
\le \varepsilon/2.
$$
For any $N\ge n_1$ and $t<T_0$ we have
$$
\int_{\{G\ge N, \Theta\le n_1\}} G\, d\mu_{t,n}
\le
(1+n_1) \int_{\{V^k\ge N/(1+n_1)\}} V^k\, d\mu_{t,n}
\le
N^{-1}(1+n_1)^2 N_{k}W_{k},
$$
which can be made smaller than $\varepsilon/2$
 uniformly in $t<T_0$
for all $N$ sufficiently large.

Thus, it remains to justify the desired convergence in the case of $f_N$, which will be now
denoted by~$f$. We recall that the restriction of such a function $f$ to every
set $\{\Theta\le R\}\times [0,T_0]$ is continuous in the first variable. Dividing by $N$ we assume that
$|f|\le 1$.
If $f$ were continuous in $x$ on the whole space, this would follow at once from
the weak convergence of $\mu_{t,n}$ for every fixed~$t$. Our situation reduces to this one
in the standard way: given~$\varepsilon>0$, we find $R$ so large that the set
$\{\Theta\le R\}\times [0,T_0]$ has measure less than $\varepsilon$ with respect to
all measures $\mu_{t,n}\, dt$ and $\mu_t\, dt$.
By our assumption the set $\Omega=\{\Theta\le R\}$ is compact in $\mathbb{R}^\infty$.
The mapping $t\mapsto f(\,\cdot\, ,t)$ from $[0,T_0]$ to $C(\Omega)$ is Borel measurable.
By Dugundji's theorem (see \cite[Chapter~III, Section~7]{BP}), there is a linear extension operator
$E\colon\, C(\Omega)\to C_b(\mathbb{R}^\infty)$ such that $E\varphi(x)=\varphi(x)$
for all $\varphi\in C(\Omega)$, $x\in \Omega$ and $\|E\varphi\|_\infty=\|\varphi\|_\infty$.
Letting $g(x,t)=Ef(\,\cdot\, ,t)(x)$, we obtain a Borel function (since it is Borel measurable
in $t$ and continuous in~$x$, see \cite[Lemma~6.4.6]{mera}) such that $|g|\le 1$ and $g(t,x)=f(t,x)$ for all $x\in\Omega$.
The integral of $g$ with respect to $\mu_{t,n}\, dt$ converges to the integral of $g$ with respect to $\mu_t\, dt$
and the integrals of $|f-g|$ with respect to these measures do not exceed~$\varepsilon$.
Therefore, the measure $\mu=\mu_t\, dt$ satisfies our
parabolic equation with initial distribution~$\nu$.
\end{proof}

The condition that $V\ge 1$ is taken just for simplicity of estimates:
it can be replaced by $V\ge 0$ if we add constants in the right sides of
(\ref{e4.1}) and (\ref{e4.2}).

In typical examples $V$ and $\Theta$ are quadratic functions (with added constants).
For example, we shall use $V(x)=\sum_{i=1}^\infty \beta_i x_i^2+1$ and
$\Theta(x)=\sum_{i=1}^\infty \alpha_i x_i^2$.
There is also a version of this theorem applicable to exponents of quadratic functions
(the first inequality in (\ref{e4.1}) is not suitable for such functions).

\begin{theorem}\label{t4.1new}
Suppose that in Theorem~{\rm\ref{t4.1}}
condition~{\rm(\ref{e4.1})} is replaced by
\begin{equation}\label{e4.1b}
LV(x,t)\le V(x)-V(x)\Theta(x)
\end{equation}
and {\rm(\ref{e4.2})} is replaced by
\begin{equation}\label{e4.2b}
|a^{ij}(x,t)|+|B^i(x,t)|\le C_i (1+\delta(V(x)\Theta(x))V(x)\Theta(x)),
\ (x,t)\in \mathbb{R}^\infty \times [0,T_0].
\end{equation}
Then, for every Borel probability measure $\mu_0$ on $\mathbb{R}^\infty$ with
$W_1:=\sup_n \|V\circ P_n\|_{L^1(\mu_0)}<\infty$
the Cauchy problem {\rm(\ref{e1})}  with initial distribution $\mu_0$ has a
solution of the form $\mu=\mu_t\, dt$ with Borel probability measures
$\mu_t$ on~$\mathbb{R}^\infty$ such that for $t\in [0,T_0]$
\begin{equation}\label{e4.3b}
\int_{\mathbb{R}^\infty} V\, d\mu_t
+\int_0^t \int_{\mathbb{R}^\infty} V \Theta\, d\mu_s\, ds
\le 4 W_1.
\end{equation}
\end{theorem}
\begin{proof}
The reasoning is much the same as in the previous theorem,
but we use only one Lyapunov function $V$ and use (\ref{e4.1b})
in place of (\ref{e4.4}) to obtain the estimate
$$
\int_{H_n} V(x)\, \mu_{t,n}(dx)
+\int_0^t \int_{H_n} V(x)\Theta(x)\, \mu_{s,n}(dx)\, ds
\le (e+1)\int_{H_n} V(x)\, \mu_{0,n}(dx)\le 4W_1.
$$
Another place where some difference arises
is the estimate of the integral of $f I_{|f|>N}$, where $|f|$ is estimated by
$C(1+\delta(V\Theta)V\Theta)$, but this is easily done by using the previous inequality and the condition that
$\delta(s)\to 0$ as $s\to \infty$.
\end{proof}

Let us apply the last theorem to the Fokker--Planck--Kolmogorov equation associated with
the stochastic Burgers type equations (see Example~\ref{ex2.10}).

\begin{example}\label{ex3.3}
\rm(``Stochastic Burgers equation'')
Let us return to the situation of Example~\ref{ex2.10}.
Let $u$ be from the linear span of $\{e_k\}$.
Note that
$$
\langle B(u), u \rangle_2=-\|u\|_{H^1_0}^2.
$$
Let $V(u)=\exp\bigl(\delta\|u\|_2^2\bigr)$.
We have
$$
LV(u)\le 2\delta\bigl({\rm tr}S+2\delta\langle Su, u\rangle_2
-\|u\|_{H_0^1}^2\bigr)V(u).
$$
Taking $\delta<\varepsilon_0/4$ we obtain
$$
LV(u)\le (1-\Theta(u))V(u), \quad \Theta(u)=1-2\delta\,{\rm tr}S +\delta\|u\|_{H_0^1}^2.
$$
In addition, $|B^k(u)|\le C(k)+C(k)\|u\|_{2}^2$. According to Theorem \ref{t4.1new}
for every initial condition $\nu$ with $\exp(\delta\|u\|_{2}^2)\in L^1(\nu)$
there exists a probability solution $\mu$ of the Cauchy problem
$\partial_t\mu=L^{*}\mu$, $\mu|_{t=0}=\nu$ such that
$$
\int_0^{T_0}\int_{L^2((0, 1))}\|u\|_{H_0^1}^2\exp\bigl(\delta\|u\|_{2}^2\bigr)\,\mu_t(du)\,dt<\infty.
$$
According to Example \ref{ex2.10} this $\mu$ is the
unique probability solution with this property.
\end{example}

\begin{example}\label{ex3.4}
\rm Let us return to the situation of Example~\ref{ex2.11}(i).
Assume that $a^{ij}=0$ if $i\ne j$ and that $\sum_{i}a^{ii}<\infty$.
Let $u$ be from the linear span of $\{e_k\}$.
Set
$$
V(u)=(1+\|u\|_{2m+2}^{2m+2})\exp\bigl(\delta\|u\|_{2}^2\bigr)
$$
Note that for some positive constants $C_1$, $C_2$ and $C_3$ we have
$$
L(1+\|u\|_{2m+2}^{2m+2})\le C_1-C_2\|u^{m+1}\|_{H_0^1}^2-C_3\|u\|_{4m+2}^{4m+2}.
$$
Using the calculations from the previous example we obtain
$$
LV(u)\le (1-\Theta(u))V(u), \quad
\Theta(u)=\widetilde{C}_1 +\delta \widetilde{C}_2\|u\|_{H_0^1}^2
+\bigl(\widetilde{C}_3\|u^{m+1}\|_{H_0^1}^2+\widetilde{C}_4\|u\|_{4m+2}^{4m+2}\bigr)
\bigl(1+\|u\|_{2m+2}^{2m+2}\bigr)^{-1}
$$
for some positive
constants $\widetilde{C}_1$, $\widetilde{C}_2$, $\widetilde{C}_3$ and $\widetilde{C}_4$.
According to Theorem \ref{t4.1new}
for every initial condition $\nu$ with
$$
(1+\|u\|_{2m+2}^{2m+2})\exp(\delta\|u\|_{2}^2)\in L^1(\nu)
$$
there exists a probability solution $\mu$ of the Cauchy problem
$\partial_t\mu=L^{*}\mu$, $\mu|_{t=0}=\nu$ such that
$$
\int_0^{T_0}\int_{L^2((0, 1))}
\bigl(\|u\|_{4m+2}^{4m+2}+\|u^2\|_{H_0^1}\bigr)
\exp\bigl(\delta\|u\|_{2}^2\bigr)\,\mu_t(du)\,dt<\infty.
$$
According to Example \ref{ex2.11}(i),
this $\mu$ is the unique probability solution with this property.

In the same way applying Theorem \ref{t4.1new} with
$$
V(u)=(1+\|u\|_2^2+\|u\|_{2l+2}^{2l+2})\exp\bigl(\delta\|u\|_{2m-2}^{2m-2}\bigr)
$$
one can obtain existence of a probability solution in the situation of Example \ref{ex2.11}(ii). Thus,
 there is a unique probability solution $\mu$ such that
$$
\int_0^{T_0}\int_{L^2((0, 1))}
\bigl(\|u\|_{4l+2}^{4l+2}+\|u\|_{H_0^1}^2+\|u^{m-1}\|_{H_0^1}^2\bigr)
\exp\bigl(\delta\|u\|_{2m-2}^{2m-2}\bigr)\,\mu_t(du)\,dt<\infty.
$$
We note only that $\|u^m\|_{H_0^1}\le \|u\|_{H_0^1}^2+\|u^{m-1}\|_{H_0^1}^2$, since
$m\ge 2$ and $\|u\|_\infty\le \|u\|_{H_0^1}$.
This partially generalizes a result in \cite{RS}.
\end{example}

Let us apply the existence theorems to the Fokker--Planck--Kolmogorov equation associated with
the stochastic Navier--Stokes equation in any dimension (a special case has been
considered in Example \ref{ex2.12}).

\begin{example}\label{ex3.5}
{\rm
The stochastic  equation of Navier--Stokes type
is considered in the space
$V_2$ of $\mathbb{R}^d$-valued mappings $u=(u^1,\ldots,u^d)$
such that $u^j\in H_0^{2,1}(D)$ and ${\rm div}\, u=0$,
where $D\subset\mathbb{R}^d$ is a bounded domain with smooth
boundary. The space $V_2$ is equipped with its natural Hilbert norm
$\|u\|_{V_2}$ defined by
$$
\|u\|_{V_2}^2:=\sum_{j=1}^d \|\nabla_z u^j\|_{2}^2.
$$
Let $H$ be the closure of $V_2$ in $L^2(D,\mathbb{R}^d)$
and let $P_H$ denote the orthogonal projection on $H$ in $L^2(D,\mathbb{R}^d)$.
The stochastic Navier--Stokes  equation is formally written as
$$
du(z,t)=\sqrt{2}dW(z,t)+P_H\Bigl[\Delta_z u(z,t)
-\sum_{j=1}^d u^j(z,t)\partial_{z_j} u(z,t)+F(z,u(z,t),t)\Bigr]dt,
$$
where $W$ is a Wiener process
of the form $W(z,t)=\sum_{n=1}^\infty \sqrt{\alpha_n} w_n(t)\eta_n(z)$,
where
$$
\alpha_n\ge 0, \quad
\sum_{n=1}^\infty \alpha_n<\infty,
$$
$w_n$ are independent Wiener processes, and $\{\eta_n\}$ is an orthonormal basis
in~$H$,  and $F\colon\, D\times\mathbb{R}^d\times (0,T_0)\to\mathbb{R}^d$ is a bounded continuous mapping.
No interpretation of this equation is needed for the sequel, it should be regarded only as a heuristic
expression leading to a specific form of the corresponding elliptic operator.
The case $F=0$ is the classical stochastic Navier--Stokes  equation.
Note that the action of $P_H$ in the right-hand side is defined in the natural way:
$P_H \Delta_z u(z,t):= P_H \Delta_z u(\,\cdot\,,t)(z)$ and similarly for the
other terms.
Since the Laplacian $\Delta$ is not defined
on all of~$V_2$, this equation requires some interpretation.
Our approach suggests the following procedure.
It is known (see \cite{Lad}) that there exists
an orthonormal basis $\{\eta_n\}$ in~$H$ formed
by eigenfunctions of $\Delta$
with eigenvalues $-\lambda_n^2$
such that $\eta_n\in V_2$.
Employing the fact that $\langle P_H w,\eta_n\rangle_2=\langle w,\eta_n\rangle_2$ for any
$w\in L^2(D,\mathbb{R}^d)$,
we introduce the ``coordinate'' functions
\begin{align*}
B^n(u,t) &=
\langle u,\Delta \eta_n\rangle_2
-\sum_{j=1}^d \langle P_H (u^j\partial_{z_j} u), \eta_n\rangle_{2}
+\langle P_H F(\,\cdot\, ,u(\,\cdot\,,t),t),\eta_n\rangle_{2}
\\
&=\langle u,\Delta \eta_n\rangle_2
-\sum_{j=1}^d \langle \partial_{z_j}u, u^j\eta_n\rangle_{2}
+\langle F(\,\cdot\, ,u(\,\cdot\,,t),t),\eta_n\rangle_{2}.
\end{align*}
These functions are defined by the last line on
all of~$V_2$. They are continuous
on balls in $V_2$ with respect to the topology of $L^2(D,\mathbb{R}^d)$, which follows
by the compactness of the embedding of $H^{2,1}(D)\to L^2(D)$.
Choosing a Wiener process of the above form, we arrive at the
operator
$$
L\varphi(u,t)=
\sum_{n=1}^\infty \alpha_n \partial_{\eta_n}^2 \varphi (u,t)
+\sum_{n=1}^\infty B^n(u,t)\partial_{\eta_n}\varphi (u,t).
$$
Since for every $u$ from the linear span of $\{\eta_n\}$ one has
$$
\sum_{n=1}^\infty \sum_{j=1}^d \langle u,\eta_n\rangle_2
\langle\partial_{z_j} u, u^j\eta_n\rangle_2
=\sum_{j=1}^d \langle u,u^j\partial_{z_j}u\rangle_2=
-\frac{1}{2}\int_D |u(z)|^2 {\rm div}\, u(z)\, dz=0
$$
and $\langle\Delta u,u\rangle_2=-\|u\|_{V_2}^2$, we have the estimate
$$
\sum_{n=1}^N \langle u,\eta_n\rangle_2 B^n(u,t)\le C_1-C_1\|u\|_{V_2}^2
$$
for all $u$ in the linear span of $\eta_1,\ldots,\eta_N$,
where $C_1$ is a constant independent of~$N$.
Clearly, we have also $|B^n(u,t)|\le C_2(n)+C_2(n)\|u\|_2^2$.

Therefore, by Theorem~\ref{t4.1} applied with $\Theta(u)=C_1\|u\|_{V_2}^2$ and $V(u)=\|u\|_2^2+1$
(the above estimates along with convergence of the series of $\alpha_n$
mean that we have (\ref{e4.1}))
there is a probability measure $\mu=\mu_t dt$ on $V_2\times [0,T_0)$,
such that $\mu_t(H)=1$ for all $t$ and $\mu_t(V_2)=1$ for almost all~$t$,
and solving the Cauchy problem (\ref{e1})
with any initial distribution
$\mu_0$ for which $\|u\|_2^k\in L^1(\mu_0)$ for all~$k$.

It should be also noted that Flandoli and Gatarek \cite{Flan} proved
(under the stated assumptions)  the existence of a solution to the martingale problem
associated with the operator $L$ such that this solution possesses all moments in~$H$.
One can show that the measure generated by this solution satisfies the Fokker--Planck--Kolmogorov
equation in our sense.

Let us consider the 2d-Navier--Stokes equation, i.e., $d=2$ and $F=0$.
Recall that
for every $u$ from the linear span of $\{\eta_n\}$ one has
$$
\sum_{n=1}^\infty \sum_{j=1}^2 \langle u,\Delta\eta_n\rangle_2
\langle \partial_{z_j} u, u^j\eta_n\rangle_2=0.
$$
Set $V(u)=\exp(\delta\|u\|_{V_2}^2)$.
Let $u$ be from the linear span of $\{\eta_n\}$. We have
$$
LV(u)=2\delta\Bigl(\sum_n\alpha_n\lambda_n^2
+2\delta\sum_{n}\alpha_n\lambda_n^4u_n^2-\sum_n\lambda_n^4u_n^2\Bigr)V(u).
$$
Assume that $\sum_{n=1}^{\infty}\alpha_n\lambda_n^2<\infty$.
Hence for sufficiently small $\delta>0$
$$
LV(u)\le (1-\Theta(u))V(u), \quad
\Theta(u)=1-\delta\sum_{n=1}^{\infty}\alpha_n\lambda_n^2+\delta\|\Delta u\|_{2}^2,
$$
where $\Theta(u)=+\infty$ if $u\not\in H^{2,2}(D)$.
According to Theorem \ref{t4.1new}
for every initial condition $\nu$ with $\exp(\delta\|u\|_{V_2}^2)\in L^1(\nu)$
there exists a probability solution $\mu$ of the Cauchy problem
$\partial_t\mu=L^{*}\mu$, $\mu|_{t=0}=\nu$ such that
$$
\int_0^{T_0}\int_{H}\bigl(1+\|\Delta u\|_{2}^2\bigr)e^{\delta\|u\|_{V_2}^2}\,\mu_t(du)\,dt<\infty.
$$
According to Example \ref{ex2.12} this $\mu$ is the
 unique probability solution with this property.
}\end{example}

Finally, we formulate one more existence and uniqueness result which is a combination of
Theorem~{\rm\ref{t4.1}} and Theorem~{\rm\ref{th1}}.

\begin{corollary}
Let $a^{ij}=0$ if $i\neq j$ and $a^{ii}=\alpha_i>0$.
Suppose that the hypotheses of Theorem~{\rm\ref{t4.1}} are fulfilled with
certain functions $V$ and $\Theta$.
If there exists a Borel mapping $F=(F_n)\colon\, \mathbb{R}^\infty\times [0,T_0]\to\mathbb{R}^\infty$
and numbers  $p>0$, $C>0$ such that
$\|F(x, t)\|_{l^2_{\alpha}}^2\le CV(x)^p\Theta(x)$
and for each natural number $n$ the difference  $B^n(x, t)-F^n(x, t)$
depends only on $t$ and $x_1, x_2,\ldots, x_n$,
then for every initial condition $\nu$ with $V\in L^k(\nu)$ for every $k\ge 1$
the class $\mathcal{P}_{\nu}$ consists of exactly one element.
\end{corollary}

\begin{example}
\rm
Let $a^{ij}=0$ if $i\neq j$ and $a^{ii}=\alpha_i>0$.
Suppose that
$$
B^n(x, t)=-\beta_nx_n+F^n(x, t), \quad \hbox{where $\beta_n>0$.}
$$
 Let $\gamma_n\in (0,+\infty)$ be such that
$$
\sum_{n=1}^{\infty}\alpha_n\gamma_n<\infty.
$$
Let
$$
V(x)=1+\sum_{n=1}^{\infty}\gamma_nx_n^2, \quad \Theta(x)=\sum_{n=1}^{\infty}\beta_n\gamma_nx_n^2.
$$
Let $c_{00}$ denote the subspace of all vectors $x\in\mathbb{R}^{\infty}$
with at most finitely many nonzero coordinates.

Suppose that a Borel mapping
$F(\,\cdot\, , \,\cdot\,)\colon\, \mathbb{R}^\infty\times [0,T_0]\to\mathbb{R}^\infty$
satisfies the following conditions: for each $t$ it is continuous in $x$ on every set $\{\Theta\le R\}$ and
there are numbers  $\varepsilon\in(0, 1)$,
$C_1>0$, $C_2>0$, and $p>0$ such that for all $t\in (0, T)$ and $x\in c_{00}$ one has
$$
\sum_{n=1}^{\infty}\gamma_nF^n(t, x)x_n\le \varepsilon\Theta(x)+C_1V(x),
\quad
\sum_{n=1}^{\infty}\alpha_n^{-1}|F^n(t, x)|^2\le C_2\left(1+\Theta(x)\right)V(x)^p,
$$
Then, for every initial condition $\nu$
with $V\in L^k(\nu)$ for every $k\ge 1$, the class $\mathcal{P}_{\nu}$
consists of exactly one element.
\end{example}

\begin{remark}
{\rm
As already noted, if the infinite-dimensional stochastic differential equation (SDE) associated
to our Fokker--Planck--Kolmogorov equation has a solution in the sense
of Stroock--Varadhan, then one gets a solution to the FPK-equation (but not vice versa).
In contrast to that, uniqueness of solutions to the martingale problem
does not imply the uniqueness of solutions to the  FPK-equation, here the converse is true.
Therefore, the existence parts in our Examples~\ref{ex3.3} -- \ref{ex3.5} can partly also be derived
by probabilistic methods. It should also be pointed out that in these examples
we always assume that $(a^{ij})$ is trace class. For existence results by
probabilistic means in case of Example~\ref{ex3.3} and the first part of Example~\ref{ex3.4}
without this condition we refer to \cite{G} and its recent improvement \cite{RZZ}.
Furthermore, we believe that by a similar method as in \cite{DD}
one can also prove uniqueness for the FPK-equation in the Burgers case (see Example~\ref{ex3.3})
without the trace class condition.
Finally, we point out that here we consider the Burgers case only on
the bounded domain $D=(0,1)\subset \mathbb{R}$. If $D=\mathbb{R}$
existence, however, also holds. This follows from the probabilistic results
in \cite{GN}.
}\end{remark}

\vspace*{0.1in}

Vladimir Bogachev:  Department of Mechanics and Mathematics, Moscow State
University, 119991 Moscow, Russia

Giuseppe Da Prato: Scuola Normale Superiore di Pisa, Pisa, Italy

Micahel R\"ockner: Fakult\"at f\"ur Mathematik, Universit\"at Bielefeld,
D-33501 Bielefeld, Germany

Stanislav Shaposhnikov: Department of Mechanics and Mathematics, Moscow State
University, 119991 Moscow, Russia

\end{document}